\documentclass[twocolumn]{autart}    

\usepackage{graphicx}          
\usepackage{amsmath}
\usepackage{amssymb}
\usepackage{float}

\usepackage{algorithm}
\usepackage{algpseudocode}
\usepackage{etoolbox}
\AtBeginEnvironment{algorithmic}{%
  \linespread{1}\selectfont%
  \setlength{\itemsep}{0pt}%
  \setlength{\parsep}{0pt}%
}
\newtheorem{Theorem}{Theorem}
\newtheorem{Lemma}{Lemma}

\newtheorem{Definition}{Definition}
\newtheorem{Corollary}{Corollary}
\newtheorem{example}{Example}
\newenvironment{proof}{\noindent\textbf{Proof.}}{\hfill$\square$}
\DeclareMathOperator*{\argmin}{arg\,min}

\usepackage{color-edits}
\addauthor{dy}{blue}

\begin{document}

\begin{frontmatter}

\title{Occupation-Measure Mean-Field Control: Optimization over Measures and Frank-Wolfe Methods\thanksref{footnoteinfo}} 
\thanks[footnoteinfo]{
Corresponding author: Chaoying Pei (cpk4t@mst.edu).}

\author[Purdue]{Di Yu}\ead{yu1128@purdue.edu},    
\author[Lilly]{Sixiong You}\ead{yousixiong@gmail.com},               
\author[MST]{Chaoying Pei}\ead{cpk4t@mst.edu}


\address[Purdue]{Department of Statistics, Purdue University, West Lafayette, IN 47907, USA}
\address[Lilly]{Eli Lilly and Company, Indianapolis, IN 46225, USA}
\address[MST]{Department of Mechanical and Aerospace Engineering, Missouri University of Science and Technology, Rolla, MO 65401, USA}
       

\begin{abstract}                          
Coordinating large populations of autonomous agents, such as UAV swarms or satellite constellations, poses significant computational challenges for traditional multi-agent control methods. This paper introduces a new optimization framework for large-population control, termed occupation-measure mean-field control (OM-MFC). The framework models the evolution of agent populations directly in the space of occupation measures and casts large-population control as an infinite-dimensional optimization problem over measures, which becomes convex under a positive-semidefiniteness condition on the interaction kernel. A Frank--Wolfe (FW) algorithm and its fully-corrective variant (FCFW) are developed to solve the resulting problem efficiently, where each iteration reduces to a classical optimal control subproblem. Theoretical results establish convexity, existence of optimal solutions, and convergence guarantees of the proposed algorithms. Owing to its measure-based formulation, the framework naturally accommodates systems with very large numbers of agents. Numerical experiments on UAV swarm coordination and satellite constellation control demonstrate the scalability and effectiveness of the proposed approach in high-dimensional and constrained environments.
\end{abstract}

\end{frontmatter}

\section{Introduction}
Recent advances in autonomous systems have enabled the deployment of large-scale multi-agent networks, including unmanned aerial vehicle (UAV) swarms, satellite constellations, and other cooperative robotic systems. As the population size increases, conventional multi-agent control frameworks encounter fundamental challenges in scalability and in coordinating interactions among numerous agents. 

These limitations have motivated the development of the mean-field control (MFC) paradigm, which models large-population coordination by treating the swarm’s time-varying state as a probability measure and steering its evolution under performance, interaction, and safety objectives \cite{elamvazhuthi2019mean}. Several lines of work have advanced this paradigm from different perspectives. Macroscopic formulations derived from stochastic processes and partial differential equations (PDEs) describe population dynamics at the distribution level and provide strong theoretical foundations, and have been applied to problems such as coverage, task allocation, and consensus \cite{fornasier2014mean}. Geometric viewpoints recast MFC within optimal transport, where distributional motion follows Wasserstein geodesics \cite{emerick2023continuum}. Learning-based approaches—notably deep reinforcement learning \cite{chen2020mean} and neural value-function surrogates \cite{ruthotto2020machine}—have demonstrated promising scalability in high-dimensional settings. Related iterative learning schemes, such as fictitious play and best-response dynamics, have also been studied extensively in the mean-field game literature for equilibrium computation under suitable assumptions \cite{cardaliaguet2017learning}. More recently, optimization-oriented approaches have also emerged, including MFC barrier functions for safety guarantees \cite{fung2025mean} and kernel-expansion methods with primal–dual schemes that yield convergent computational procedures \cite{vidal2025kernel}.
Taken together, these developments highlight significant progress in both modeling and computation for large-population control. At the same time, many existing approaches either rely on discretization of high-dimensional state spaces or involve inherently nonconvex formulations, which can limit scalability or theoretical tractability in large-scale settings. This motivates the development of alternative frameworks that preserve rigorous structure while remaining computationally scalable for large populations.

To address this gap, we introduce a new framework -- occupation-measure mean-field control (OM-MFC) -- for large-population control in the space of occupation measures. The framework is inspired by occupation-measure techniques that have been widely used in classical optimal control \cite{lasserre2008nonlinear}. In classical optimal control, occupation measures provide a powerful lifting of trajectory-based problems into the space of measures, where the dynamics become linear constraints and the formulation becomes amenable to convex relaxation and global analysis. Closely related occupation-measure ideas have also appeared in the mean-field games literature. For example, Hu et al. develop an online reinforcement learning framework for finite-state mean-field games in which a mean-field occupation measure represents the joint state–action distribution of the population and serves as the main decision variable for computing approximate Nash equilibria \cite{hu2025mf}. Other works use occupation measures to study self-interacting Brownian systems and McKean–Vlasov processes, where the measure encodes long-time path statistics and supports large-deviation analysis \cite{konig2017mean,du2025self}. 

In contrast to these lines of work, the proposed OM-MFC framework casts large-population control as a population-level optimization over occupation measures subject to linear dynamical constraints. This leads to an infinite-dimensional optimization problem over measures. Unlike prior measure-based approaches that rely on finite-dimensional relaxations or grid approximations, the proposed framework operates directly in the measure space, enabling efficient first-order algorithms and natural scalability to large agent populations.

Optimization over measures has a long history in statistics, signal processing, and machine learning~\cite{1960kie,2017boygeorec,2018mei}. Such formulations provide a natural way to represent distributions, mixtures, and population-level behaviors, and allow certain nonlinear or combinatorial problems to be expressed as convex programs in measure spaces. Existing approaches often approximate the infinite-dimensional problem in finite dimensions, which compromises accuracy and obscures its intrinsic structure. Grid-based discretizations, in particular, suffer from exponentially increasing cost with finer resolution, reflecting the curse of dimensionality~\cite{2017boygeorec,yu2025deterministic}. To {avoid finite-dimensional discretization} while maintaining scalability, we adopt an occupation-measure relaxation that formulates the problem directly in the space of measures. In this formulation, the first variation of the objective admits a closed-form expression, enabling the use of efficient first-order algorithms—in particular, the Frank–Wolfe (FW) method—with provable convergence guarantees in the infinite-dimensional setting.

Recent advances, building on the classical FW framework~\cite{jaggi2013revisiting}, have demonstrated that FW algorithms are well suited for optimization over measures~\cite{2017boygeorec,yu2025deterministic,yu2025derivative,yu2025frank}. Related generalized conditional-gradient methods have also been studied in infinite-dimensional and mean-field settings, particularly at the PDE or potential-game level~\cite{lavigne2023generalized}. Their projection-free nature makes them particularly effective in our setting, because under the OM-MFC formulation the dynamical constraints become linear in the occupation measures. This reformulation makes the FW framework especially natural, while its operation directly in the infinite-dimensional measure space eliminates the need for discretization. Building on these properties, we develop an FW algorithm within the OM-MFC, together with a fully-corrective Frank-Wolfe (FCFW), where problem~\eqref{mainObj} is solved iteratively through a sequence of classical optimal control subproblems. Each subproblem can be efficiently handled using existing optimal control techniques, resulting in a simple yet scalable procedure that preserves feasibility, achieves the standard $\mathcal{O}(1/k)$ convergence rate, and provides closed-form, trajectory-based updates.


OM-MFC can be viewed as a population-level relaxation of the finite-agent interacting control problem and, at the PDE level, as a relaxed-control generalization of classical MFC. This viewpoint is useful not only for computational tractability, but also because it preserves the population-level dynamical structure of the original problem and remains closely connected to classical MFC at the PDE level. A preliminary version of this work appeared in~\cite{yuyoupei2026ACC}.
The main contributions of this work are summarized as follows:
\begin{itemize}
    \item \textbf{OM-MFC framework:} We introduce OM-MFC, a direct optimization-over-measures framework for large-population control. The formulation is posed as an optimization over occupation measures subject to linear dynamical constraints, and we establish its convexity, weak$^*$ compactness, and existence of optimal solutions. 
    \item \textbf{Infinite-dimensional solution algorithm:} We develop FW and FCFW algorithms that operate directly in the space of measures. We show that the linear minimization subproblem reduces to a family of classical optimal control problems parameterized by the initial condition, yielding a projection-free method with $\mathcal{O}(1/k)$ convergence guarantees and finite mixtures of admissible classical trajectories.
    \item \textbf{Scalability and validation:} Numerical experiments show qualitative agreement with PDE-based references in two-dimensional settings, and demonstrate scalability to more complex three-dimensional problems, including swarm coordination with multiple obstacles and satellite-constellation control.
\end{itemize}

Building on these properties, we develop a Frank--Wolfe (FW) algorithm within the OM-MFC framework, together with a fully-corrective Frank--Wolfe (FCFW) variant, to solve problem~\eqref{mainObj} through a sequence of classical optimal control subproblems. This yields a projection-free and scalable procedure that preserves feasibility, achieves the standard $\mathcal{O}(1/k)$ convergence rate, and produces trajectory-based updates.

The remainder of this paper is structured as follows.
Section~\ref{sec:motivation} introduces the finite-agent motivation and its occupation-measure representation.
Section~\ref{sec:main_problem} presents the OM-MFC framework as a measure-theoretic relaxation.
Section~\ref{sec:properties} analyzes the fundamental properties of the resulting measure optimization problem.
Section~\ref{sec:fw_methods} develops FW methods for the proposed framework and establishes their theoretical guarantees.
Section~\ref{simulation} reports numerical results in both two- and three-dimensional scenarios, including swarm and satellite-constellation applications.
Finally, Section~\ref{conclusion} concludes the paper and discusses future research directions.

\section{Motivation: The N-Agent Problem}\label{sec:motivation}
To motivate the OM-MFC framework, we begin by introducing the classical $N$-agent population model, which captures the discrete dynamics of large-scale multi-agent systems. We formulate the corresponding finite-agent cost functional and show how it can be expressed in terms of empirical occupation measures. This representation provides the foundation for the measure-theoretic formulation developed in Section~\ref{sec:main_problem}.

\subsection{$N$-Agent Dynamics and Cost Functional}

We model a large population of controlled systems indexed by $i=1,\dots,N$ over a finite horizon $[0,T]$, where the $i$-th unit has state $x_i(t)\in\mathbb{R}^{n_x}$ and control input $u_i(t)\in\mathbb{R}^{n_u}$.  
The evolution of each agent is governed by a common nonlinear dynamics map $f:\mathbb{R}^{n_x}\times\mathbb{R}^{n_u}\to\mathbb{R}^{n_x}$,
\begin{equation}
    \dot{x}_i(t) = f\big(x_i(t),u_i(t)\big), \quad t\in[0,T],\; i=1,\dots,N,
\end{equation}
where $f$ is continuous and Lipschitz in $x$ uniformly in $u$.
The agents are assumed identical and interchangeable, and the empirical distribution of their states at time $t$ is described by
\begin{equation}
    \rho_t^N := \frac{1}{N}\sum_{i=1}^N \delta_{x_i(t)}.
\end{equation}
A broad range of mission objectives and interaction effects can be captured through the $N$-agent cost
\begin{align}\label{eq:NagentMF}
J_N &= \frac{1}{N}\sum_{i=1}^N \int_0^T 
\Big( \ell_0\big(t,x_i(t),u_i(t)\big)
+ \lambda \,(W*\rho_t^N)\big(x_i(t)\big)\Big)\,dt \nonumber\\
&\quad + \int_{\mathcal X} \Psi(x)\,d\rho_T^N(x),
\end{align}
where $\ell_0$ encodes individual performance criteria (such as control effort, tracking error, or energy use), $W$ is an interaction kernel capturing pairwise influences including collision avoidance, cohesion, or formation keeping, and $\Psi$ specifies terminal mission objectives (for example, target-region reachability or desired spatial configurations).  
This cost function is sufficiently general to represent many large-population tasks considered in MFC \cite{carmona2018probabilistic}.

\subsection{Measure Representation of Single Trajectories}
Let
\[
\Sigma=[0,T]\times\mathcal X\times\mathcal U,
\]
where $\mathcal X\subset\mathbb R^{n_x}$ and $\mathcal U\subset\mathbb R^{n_u}$ are compact state and control sets. We write $\mathcal M(\Sigma)$ for the space of finite signed Borel measures on $\Sigma$, $\mathcal M_+(\Sigma)$ for its nonnegative cone, and
\[
\mathcal M_+(\Sigma;T):=\{\mu\in\mathcal M_+(\Sigma):\mu(\Sigma)=T\}.
\]
Similarly, $\mathcal M_+(\mathcal X;1)$ denotes the set of probability measures on $\mathcal X$. For any admissible trajectory $\omega$ in $\Gamma$
\begin{align}\label{eq:trajSpace}
\Gamma
:= &\Big\{
\omega = (x(\cdot),u(\cdot)) :\;
x \in AC([0,T];\mathcal X),\; \nonumber\\
& \quad u: [0,T] \to \mathcal U \text{ is measurable},\;\nonumber\\
& \quad \dot x(t)=f(x(t),u(t))\ \text{for a.e. } t\in[0,T]
\Big\},
\end{align}
we define the associated \emph{occupation measures} $(\mu[\omega],\nu[\omega])$. The \emph{running occupation measure} $\mu[\omega]\in\mathcal{M}_+(\Sigma,T)$ is characterized by, for any $\varphi\in C(\Sigma)$,
\begin{equation}\label{eq:runningMeasure}
\int_\Sigma \varphi(t,x,u)\,d\mu[\omega](t,x,u) 
= \int_0^T \varphi(t,x(t),u(t))\,dt.
\end{equation} 
The \emph{terminal occupation measure} 
$\nu[\omega]\in \mathcal{M}_+(\mathcal{X},1)$ is given by \(\nu[\omega]=\delta_{x(T)},\) or equivalently,
\begin{equation}\label{eq:terminalMeasure}
\int_{\mathcal X}\psi(x)\,d\nu[\omega](x)=\psi(x(T)),
\qquad \forall \psi\in C(\mathcal X).
\end{equation}
Thus $(\mu[\omega],\nu[\omega])$ provides a measure-theoretic representation of the trajectory $\omega$. Averaging such single-trajectory measures across agents yields the empirical occupation measures used below to rewrite the $N$-agent cost.

\subsection{Reformulation via Occupation Measures}
The following theorem states the occupation–measure representation of the N-agent cost; a related presentation appears in~\cite{yuyoupei2026ACC}.
\begin{Theorem}[Occupation–measure form of $J_N$]\label{Thm:objJN}
Consider agents $i=1,\dots,N$ with trajectories $\omega_i$ and associated occupation measures $(\mu[\omega_i],\nu[\omega_i])$.
Define the empirical averages
$\mu^N:=\tfrac1N\sum_{i=1}^N \mu[\omega_i]\in\mathcal M_+(\Sigma,T)$ and
$\nu^N:=\tfrac1N\sum_{i=1}^N \nu[\omega_i]\in\mathcal M_+(\mathcal X,1)$.
Suppose that $\ell_0$ and $\Psi$ are bounded continuous, and that $W$ is a measurable function with finite integrals.
Then the cost functional \eqref{eq:NagentMF} of the $N$-agent system can be expressed as
\begin{align}\label{eq:JN-measure}
J_N &= \int_{\Sigma} \ell_0(t,x,u)\,d\mu^N(t,x,u) \nonumber\\
& + \lambda\!\!\iint_{\Sigma\times\Sigma}\!\! W(x-y)\,\delta(t-t')\,d\mu^N(t,x,u)\,d\mu^N(t',y,v) \nonumber\\
& + \int_{\mathcal X}\Psi(x)\,d\nu^N(x).
\end{align}
Here the symbol $\delta(t-t')$ denotes the Dirac distribution and enforces $t=t'$ in the interaction term.

In addition, the pair $(\mu^N,\nu^N)$ satisfies the following averaged weak Liouville relation:
for every $v\in C^1([0,T]\times\mathcal X)$,
\begin{align}\label{eq:weak-Liouville-N}
\int_{\mathcal X} v(T,x)\,d\nu^N(x)
&- \int_{\mathcal X} v(0,x)\,d\rho_0^N(x) \nonumber\\
&= \int_{\Sigma}\big(\partial_t v+\nabla_x v\!\cdot\! f(x,u)\big)\,d\mu^N,
\end{align} where $\rho_0^N = \tfrac1N\sum_{i=1}^N \delta_{x_i(0)}$ denotes the initial empirical distribution.
\end{Theorem}
\begin{proof}
Decompose the cost $J_N$ in~\eqref{eq:NagentMF} into the running, interaction, and terminal terms. By the definition of the running occupation measure~\eqref{eq:runningMeasure},
\[
\frac{1}{N}\sum_{i=1}^N \int_0^T \ell_0(t,x_i(t),u_i(t))\,dt
=
\int_{\Sigma} \ell_0(t,x,u)\,d\mu^N(t,x,u).
\]
For the interaction term, using $\mu^N=\tfrac1N\sum_{i=1}^N\mu[\omega_i]$, the definition of the occupation measures and $\rho_t^N=\tfrac1N\sum_{j=1}^N\delta_{x_j(t)}$,
\begin{align*}
    \lambda\iint &W(x-y)\,\delta(t-t')\,d\mu^N(t,x,u)\,d\mu^N(t',y,v) \\&\quad=\lambda\int_0^T \frac{1}{N^2}\sum_{i,j=1}^N
W\bigl(x_i(t)-x_j(t)\bigr)\,dt \\&\quad = \frac{\lambda}{N}\sum_{i=1}^N \int_0^T (W*\rho_t^N)(x_i(t))\,dt.
\end{align*}
Similarly, by~\eqref{eq:terminalMeasure},
\(
\int_{\mathcal X}\Psi(x)\,d\nu^N(x)
=
\frac1N\sum_{i=1}^N \Psi(x_i(T)).
\)
Combining the three identities yields~\eqref{eq:JN-measure}. To prove~\eqref{eq:weak-Liouville-N}, fix $v\in C^1([0,T]\times\mathcal X)$. For each agent,
\[
\frac{d}{dt}v(t,x_i(t))
=
\partial_t v(t,x_i(t))+\nabla_x v(t,x_i(t))\cdot f(x_i(t),u_i(t)).
\]
Integrating on $[0,T]$ and averaging over $i=1,\dots,N$ gives
\begin{align*}
\frac1N\sum_{i=1}^N &v(T,x_i(T))
-
\frac1N\sum_{i=1}^N v(0,x_i(0))
\\&\qquad=
\frac1N\sum_{i=1}^N \int_0^T
\bigl(\partial_t v+\nabla_x v\cdot f(x_i,u_i)\bigr)\,dt.
\end{align*}
Using~\eqref{eq:runningMeasure}, \eqref{eq:terminalMeasure}, and
$\rho_0^N=\tfrac1N\sum_{i=1}^N\delta_{x_i(0)}$ gives~\eqref{eq:weak-Liouville-N}.
\end{proof}

The averaged occupation measures $(\mu^N,\nu^N)$ encode the
same distributional information as the empirical measures $\rho_t^N$ and
$\rho_T^N$ in the $N$-agent formulation~\eqref{eq:NagentMF}. In particular, $d\mu^N(t,x,u)=dt\,\mu_t^N(dx,du)$ with
$\mu_t^N=\tfrac1N\sum_{i=1}^N\delta_{(x_i(t),u_i(t))}$,
and the state marginal recovers the empirical distribution
\[
\rho_t^N = \pi_x\#\mu_t^N
= \tfrac{1}{N}\sum_{i=1}^N \delta_{x_i(t)},
\]
where $\pi_x$ denotes the projection onto the state component. Similarly, the terminal measure $\nu^N$ coincides with the empirical terminal distribution $\rho_T^N$. Accordingly, the interaction term in \eqref{eq:JN-measure} is understood rigorously as
\begin{equation}\label{eq:interactionRig}
\lambda \int_0^T \iint_{\mathcal X\times\mathcal U\times\mathcal X\times\mathcal U}
W(x-y)\,d\mu_t^N(x,u)\,d\mu_t^N(y,v)\,dt,
\end{equation}
and the notation involving $\delta(t-t')$ is used as a formal shorthand for same-time interactions.

Thus both the cost and the averaged dynamics can be expressed entirely in terms of $(\mu^N,\nu^N)$ through \eqref{eq:JN-measure} and \eqref{eq:weak-Liouville-N}. This empirical measure representation motivates the convex relaxation introduced in the next section.

\section{The OM-MFC Framework: A Measure-Theoretic Relaxation}\label{sec:main_problem}

\subsection{Problem Statement}
\label{subsec:formulation}
While Theorem~\ref{Thm:objJN} provides an empirical measure representation for any finite $N$-agent system, optimizing directly over the space of $N$-agent trajectories remains a computationally intensive nonconvex task. To achieve scalability and analytical tractability, we propose a \emph{convex relaxation} by extending the feasible set to the space of measures that satisfy the ensemble dynamics. Specifically, let $\rho_0 \in \mathcal{M}_+(\mathcal{X},1)$ be a prescribed initial distribution. We define the feasible set $\Delta \subset \mathcal M_+(\Sigma,T)\times\mathcal M_+(\mathcal X,1)$ as the collection of all measure pairs $(\mu,\nu)$ satisfying the weak Liouville constraint: for all test functions $v\in C^1([0,T]\times\mathcal X)$,
\begin{align}\label{eq:feasibleSet}
&\int_{\mathcal X} v(T,x)\,d\nu(x) - \int_{\mathcal X} v(0,x)\,d\rho_0(x) \nonumber \\&\qquad =\int_{\Sigma} \Big(\partial_t v+\nabla_x v\cdot f(x,u)\Big)\,d\mu(t,x,u).
\end{align}
By lifting the problem from discrete trajectories to the convex set $\Delta$, we define the \emph{Occupation-Measure Mean-Field Control} (OM-MFC) problem as:
\begin{align}\label{mainObj}
\min_{(\mu,\nu)\in\Delta} \quad & 
J(\mu,\nu) :=
\int_{\Sigma} \ell_0(t,x,u)\,d\mu(t,x,u) \nonumber\\
& + \lambda \iint
W(x-y)\,\delta(t-t')\,d\mu(t,x,u)\,d\mu(t',y,v) \nonumber\\
& + \int_{\mathcal X} \Psi(x)\,d\nu(x). \tag{OM-MFC}
\end{align}
Here the interaction term is understood in the same rigorous sense as in \eqref{eq:JN-measure} and \eqref{eq:interactionRig}. The formulation \eqref{mainObj} represents an infinite-dimensional program over the space of Borel measures. The constraint set $\Delta$ is convex, and the objective becomes convex under suitable assumptions on the interaction kernel, supporting the use of projection-free solvers like the FW method.

\subsection{Relaxation and Relation to Classical MFC}\label{relation}
\label{subsec:discussion}

Problem~\eqref{mainObj} provides a measure-theoretic generalization of classical MFC formulations. The difference originates already at the $N$-agent level: in~\eqref{eq:NagentMF}, each agent is allowed to apply its own control input $u_i(t)$, without restriction to a common deterministic feedback policy of the form $u_i(t)=\phi(t,x_i(t))$ as typically assumed in classical MFC \cite{carmona2018probabilistic}. The associated averaged occupation measure $\mu^N$ therefore encodes \emph{heterogeneous} control actions across the population. Our measure formulation preserves this generality and does not impose a deterministic feedback structure on the control.

This modeling choice leads to a generalized population-level PDE. As shown in Appendix~\ref{sec:PDE}, any feasible $(\mu,\nu)\in\Delta$ induces a state marginal $\rho_t$ satisfying the continuity equation
\begin{equation}\label{eq:relaxed_pde}
    \partial_t \rho_t + \nabla_x \cdot \big( F(t,x) \rho_t \big) = 0,
\end{equation}
where $F(t,x) = \int_{\mathcal U} f(x,u) \, d\lambda_{t,x}(u)$ is the averaged dynamics under a state-time dependent control distribution $\lambda_{t,x}$. This equation coincides with the classical MFC PDE only in the special case where $\lambda_{t,x} = \delta_{\phi(t,x)}$, corresponding to a deterministic feedback control $u = \phi(t,x)$. Such relaxed formulations, where controls are replaced by probability measures on the action space, are a cornerstone of existence theory and limit analysis in optimal control and mean-field systems \cite{fattorini1999infinite,lacker2017limit}.

While this framework generalizes classical MFC by allowing heterogeneous actions, it also facilitates a \emph{convex relaxation} of the original $N$-agent problem. Since the velocity field $F(t,x)$ defined in \eqref{eq:relaxed_pde} is, by construction, a convex combination of classical velocities from the set $\{f(x,u) : u \in \mathcal{U}\}$, it can take any value in the convexified velocity set
\begin{equation}
    \operatorname{co} f(x, \mathcal{U}) := \left\{ \int_{\mathcal{U}} f(x,u) \, d\lambda(u) : \lambda \in \mathcal{P}(\mathcal{U}) \right\}.
\end{equation}
This property is a standard result in the theory of relaxed controls and differential inclusions, where the set of admissible velocities is extended to its closed convex hull to ensure the convexity of the feasible set \cite{filippov2013differential,vinter2010optimal}. As a result, the feasible set $\Delta$ becomes a relaxation of the original control system and enlarges the search space to the entire convex hull. Consequently, it may include trajectories that are not feasible under the original dynamics, enabling the representation of collective behaviors such as the trajectory splitting illustrated in the following example.

\begin{example}[Ghost Relaxation] \label{ex:chattering}
Consider a 1D system where agents at the same location $x$ move either left ($u=-1$) or right ($u=1$). Our framework captures this "split" through the relaxed control in \eqref{eq:relaxed_pde} (e.g., $\lambda_{t,x} = 0.5\delta_{-1} + 0.5\delta_{1}$), yielding an effective velocity $F(t,x)=0 \in \operatorname{co}\{-1, 1\}$. This $F=0$ represents a \emph{ghost relaxation}: a solution that is mathematically feasible in the convexified space $\Delta$ but is not directly available in the original dynamics. However, such ghost velocities can be realized as the limit of classical trajectories that rapidly switch (chatter) between the admissible control actions~\cite{vinter2010optimal}.
\end{example}

Despite the expansion of the feasible set to $\Delta$, the potential relaxation gap between~\eqref{mainObj} and the original $N$-agent problem~\eqref{eq:NagentMF} is addressed as follows.

\begin{enumerate}
\item The formulation~\eqref{mainObj} can be interpreted as a population-level relaxation of the $N$-agent system in the large-population regime. Rather than optimizing over $N$ coupled trajectories, we optimize over measure pairs satisfying ensemble dynamics. This viewpoint is consistent with classical mean-field limit theory for interacting particle systems~\cite{carmona2018probabilistic,lacker2017limit}.

\item Although $\operatorname{co} f(x,\mathcal U)$ allows relaxed velocities such as in Example~\ref{ex:chattering}, Lemma~\ref{lem:approximation} together with~\cite[Theorem~2.7.2]{vinter2010optimal} ensures that any relaxed trajectory can be uniformly approximated by classical admissible trajectories in $\Gamma$. This guarantees that the linear minimization oracle in the FW method can be implemented over $\Gamma$ while attaining the same oracle value as the relaxed subproblem in~\eqref{mainObj}; see Theorem~\ref{thm:lmo}.

\item Each FW iteration generates a classical admissible trajectory, and every iterate $\mu_k$ is a finite convex combination of occupation measures induced by such trajectories. Hence all iterates remain within the class of admissible classical dynamics. Under the assumptions of Theorem~\ref{thm:fw-rate}, $J(\mu_k,\nu_k)\to J^*=\min_{\Delta}J$ at rate $\mathcal O(1/k)$, yielding a sequence of finite mixtures of admissible trajectories whose objective values approach the relaxed optimum. Related randomized relaxation and relaxation-gap analyses for large-scale aggregative nonconvex problems can be found in~\cite{bonnans2023large}.
\end{enumerate}
{Accordingly, OM-MFC is interpreted here as a population-level convex relaxation that preserves the population-level dynamical structure of the original problem and remains closely connected to classical MFC at the PDE level.}

\section{Properties of the Measure Optimization Problem}\label{sec:properties}
We begin by identifying a natural condition on the interaction kernel under which the OM-MFC objective becomes convex. Specifically, we call a symmetric kernel $W:\mathbb{R}^{n_x}\to\mathbb{R}$ positive semidefinite if
\[
\iint_{\mathbb{R}^{n_x}\times\mathbb{R}^{n_x}} W(x-y)\,d\sigma(x)\,d\sigma(y)\ge 0
\]
for every finite signed measure $\sigma$ for which the integral is well defined.
\begin{Theorem}[Convexity and compactness]\label{thm:conv}
Problem
~\eqref{mainObj} satisfies:
\begin{enumerate}
\item[(a)] When the interaction kernel $W:\mathbb{R}^{n_x}\to\mathbb{R}$ is positive semidefinite, the objective $J(\mu,\nu)$ is convex over $\mathcal{M}_+(\Sigma,T)\times\mathcal{M}_+(\mathcal{X},1)$.
\item[(b)] The constraint set $\Delta$ forms a convex, weak$^*$ compact subset of $\mathcal{M}_+(\Sigma,T)\times\mathcal{M}_+(\mathcal{X},1)$.
\end{enumerate}
\end{Theorem}
\begin{proof}
The running cost and terminal cost in $J(\mu,\nu)$ are manifestly linear in $(\mu,\nu)$. Thus convexity hinges on the interaction term
\begin{align*}
&F(\alpha\mu_1+(1-\alpha)\mu_2) - (\alpha F(\mu_1)+(1-\alpha)F(\mu_2)) \\
&\quad= -\alpha(1-\alpha)\iint W\,\delta\,d(\mu_1-\mu_2)\,d(\mu_1-\mu_2) \leq 0,
\end{align*}
where the final inequality leverages the positive semidefiniteness of $W$. This establishes convexity of $F$, hence of $J$, proving (a).

For (b), the product space $\mathcal{M}_+(\Sigma,T)\times\mathcal{M}_+(\mathcal{X},1)$ is convex, while the Liouville constraint~\eqref{eq:feasibleSet} remains linear in $(\mu,\nu)$. Convexity of $\Delta$ follows immediately. weak$^*$ compactness arises since $\mathcal{M}_+(\Sigma,T)$ and $\mathcal{M}_+(\mathcal{X},1)$ are themselves weak$^*$ compact, and the defining Liouville functional is weak$^*$ continuous.
\end{proof}

Typical examples of positive semidefinite interaction kernels include Gaussian kernels such as \(W(r)=\exp\!\{-r^{2}/(2\sigma^{2})\}\),
which are also used in our numerical experiments. More generally, translation-invariant kernels with nonnegative Fourier transforms are positive semidefinite. Theorem~\ref{thm:conv} then implies the existence result stated in Corollary~\ref{Cor:exist}.
\begin{Corollary}[Solution Existence]\label{Cor:exist}
Under the assumptions of Theorem~\ref{thm:conv}, problem~\eqref{mainObj} admits at least one optimal solution 
$(\mu^*,\nu^*) \in \Delta$.
\end{Corollary}
\begin{proof}
    Let $J^* := \inf_{(\mu,\nu)\in\Delta} J(\mu,\nu)$.
Then there exists a minimizing sequence $\{(\mu_n,\nu_n)\}\subset\Delta$ such that
$J(\mu_n,\nu_n)\to J^*$. By Theorem~\ref{thm:conv}(b), $\Delta$ is weak$^*$ compact, hence (up to a subsequence) $(\mu_n,\nu_n)\rightharpoonup^*(\mu^*,\nu^*)\in\Delta$.

Since $\ell_0$ and $\Psi$ are bounded continuous, the linear terms
$\int_\Sigma \ell_0\,d\mu$ and $\int_{\mathcal X}\Psi\,d\nu$
are weak$^*$ continuous.
Moreover, the interaction functional is weak$^*$ lower semicontinuous on $\Delta$
because $W$ is continuous and positive semidefinite.
Therefore $J$ is weak$^*$ lower semicontinuous on $\Delta$, and thus
\[
J(\mu^*,\nu^*) \le \liminf_{n\to\infty} J(\mu_n,\nu_n)=J^*.
\]
Since $(\mu^*,\nu^*)\in\Delta$, we conclude $J(\mu^*,\nu^*)=J^*$.
\end{proof}

\begin{Theorem}[G\^ateaux Differentiability]\label{thm:firstVariation}
Let $(\mu,\nu)\in\Delta$. Consider signed measures $\delta\mu\in\mathcal{M}(\Sigma)$ and $\delta\nu\in\mathcal{M}(\mathcal{X})$ satisfying total mass preservation $\delta\mu(\Sigma)=\delta\nu(\mathcal{X})=0$ and the linearized Liouville equation
\begin{equation}\label{eq:linLiouville}
\int_\Sigma(\partial_tv+\nabla_xv\cdot f)\,d\delta\mu(t,x,u)=\int_{\mathcal{X}}v(T,x)\,d\delta\nu(x),
\end{equation}
for all $v\in C^1([0,T]\times\mathcal{X})$.
Then the G\^ateaux derivative of $J$ at $(\mu,\nu)$ in direction $(\delta\mu,\delta\nu)$ exists:
\begin{align}\label{eq:firstVariation}
\delta J(\mu,\nu;\delta\mu,\delta\nu)&:=\left.\frac{d}{d\varepsilon}\,J(\mu+\varepsilon\delta\mu,\nu+\varepsilon\delta\nu)
\right|_{\varepsilon=0} \nonumber\\
&=:\langle g_{\mu},\delta\mu\rangle + \langle \Psi,\delta\nu\rangle,
\end{align}
where
$\langle g_{\mu},\delta\mu\rangle:=\int_\Sigma g_{\mu}(t,x,u)\,d\delta\mu(t,x,u)$ and
$\langle \Psi,\delta\nu\rangle:=\int_{\mathcal X}\Psi(x)\,d\delta\nu(x)$, with
\begin{equation}\label{eq:gmu}
g_{\mu}(t,x,u)=\ell_0(t,x,u)+2\lambda\!\!\int
W(x-y)\,\delta(t-t')\,d\mu(t',y,u').
\end{equation}
\end{Theorem}
\begin{proof}
The running and terminal costs in $J(\mu,\nu)$ contribute the linear terms $\langle \ell_0,\delta\mu\rangle$ and $\langle \Psi,\delta\nu\rangle$.

For the interaction functional 
\[
F(\mu):=\iint W(x-y)\,\delta(t-t')\,d\mu(t,x,u)\,d\mu(t',y,u'),
\]
direct expansion yields
\begin{align*}
&\lim_{\varepsilon\to 0}\frac{F(\mu+\varepsilon\delta\mu)-F(\mu)}{\varepsilon} \\
&\quad= 2 \iint W\,\delta\,d\mu\,d\delta\mu 
  + \lim_{\varepsilon\to 0}\varepsilon \iint W\,\delta\,d\delta\mu\,d\delta\mu \\
&\quad= 2 \iint W\,\delta\,d\mu\,d\delta\mu.
\end{align*}
The full directional derivative~\eqref{eq:firstVariation} follows by superposition.
\end{proof}

\begin{Theorem}[Optimality Condition]\label{thm:optcond}
Let $J$ be convex on $\Delta$, and let $(\mu^*,\nu^*)\in\Delta$.
Then $(\mu^*,\nu^*)$ is optimal for~\eqref{mainObj} if and only if
\begin{equation}\label{eq:VI}
\langle g_{\mu^*},\mu-\mu^*\rangle+\langle \Psi,\nu-\nu^*\rangle \ge 0,
\qquad \forall (\mu,\nu)\in\Delta,
\end{equation}
where $g_{\mu^*}$ is defined in~\eqref{eq:gmu}.
\end{Theorem}
\begin{proof}
For any $(\mu,\nu)\in\Delta$,
 by convexity of $J$ on $\Delta$ 
\[
J(\mu,\nu)-J(\mu^*,\nu^*) \ge
\delta J(\mu^*,\nu^*;\mu-\mu^*,\nu-\nu^*).
\]
Using Theorem~\ref{thm:firstVariation} gives \eqref{eq:VI}.
Conversely, if \eqref{eq:VI} holds, then the above inequality implies
$J(\mu,\nu)\ge J(\mu^*,\nu^*)$ for all $(\mu,\nu)\in\Delta$, which proves optimality.
\end{proof}

These results establish problem~\eqref{mainObj} as a well-defined optimization problem in infinite-dimensional measure space, with convex structure under the assumptions of Theorem~\ref{thm:conv}. Additionally, the linearity of the first variation of $J$ with respect to $(\delta\mu,\delta\nu)$ provides the first-order information needed to form linearized approximations. The following section develops FW methods tailored to this measure-space setting.

\section{Frank–Wolfe Methods for Optimization over Measures}\label{sec:fw_methods}
In this section, we apply FW methods to solve problem~\eqref{mainObj}. The FW method (also known as the \emph{conditional gradient} method \cite{2015bub,jaggi2013revisiting}), is particularly well-suited for optimization over measures \cite{yu2025deterministic}, since it avoids 
projections onto the infinite-dimensional feasible set $\Delta$ by replacing 
them with a sequence of linear minimization oracles (LMO). 
This projection-free structure makes FW especially attractive in the present 
setting. In addition to the basic FW scheme, we also consider the 
fully-corrective variant, which periodically re-optimizes the weights over 
the set of previously generated occupation measures associated with 
admissible trajectories \cite{yu2025deterministic,yu2025frank}.

To motivate the extension to measure spaces, recall the classical finite-dimensional FW iteration minimizes smooth \(f:\mathbb{R}^d\to\mathbb{R}\) over compact convex \(Z\subset\mathbb{R}^d\) through
\begin{equation} \label{eq:FW}
\begin{split}
    y_{k+1} &= (1 - \alpha_k)y_k + \alpha_k s_k, \\
    s_k &:= \argmin_{s \in Z} \nabla f(y_k)^\top s,
\end{split}
\end{equation}
with step size \(\alpha_k\in(0,1]\). The method remains projection-free and maintains feasibility of all iterates \(\{y_k\}\); its effectiveness hinges on the tractability of the linear minimization subproblem.

To adapt FW to problem~\eqref{mainObj}, we employ the first variation $\delta J$ as the linearization tool. For iterate $(\mu_k,\nu_k)\in\Delta$,
\[
J(\mu,\nu) \approx J(\mu_k,\nu_k) + \delta J(\mu_k,\nu_k;\,\mu-\mu_k,\nu-\nu_k).
\]
This linearization motivates the measure-space FW recursion:
\begin{align}\label{eq:FWrec}
&(\mu_{k+1},\nu_{k+1}) = (1-\alpha_k)(\mu_k,\nu_k) + \alpha_k (\tilde{\mu}_k,\tilde{\nu}_k), \nonumber\\
&(\tilde{\mu}_k,\tilde{\nu}_k) \in \argmin_{(\tilde{\mu},\tilde{\nu})\in \Delta} \, \langle g_{\mu_k},\tilde{\mu}-\mu_k\rangle + \langle \Psi,\tilde{\nu}-\nu_k\rangle,
\end{align}
with $g_{\mu_k}$ from~\eqref{eq:gmu}. This update ensures feasibility at each step and circumvents projections onto the infinite-dimensional set $\Delta$, rendering FW methods particularly suitable here. Since the subproblem consists of linear minimization over the convex set $\Delta$, Theorem~\ref{thm:lmo} shows that an optimal solution can be constructed by selecting classical admissible trajectories and aggregating their occupation measures under the initial distribution $\rho_0$. The proof is provided in Appendix~\ref{sec:appendix_lmo_proof}.

\begin{Theorem}[Solution to FW Subproblem]\label{thm:lmo}
Fix $(\mu_k,\nu_k)\in\Delta$, where $\Delta$ is defined with initial distribution $\rho_0$.
Assume that for $\rho_0$-almost every $\xi\in\mathcal X$, the parametric problem
\begin{equation}\label{eq:equivFWSub}
V(\xi):=\inf_{\substack{\dot x=f(x,u)\\ x(0)=\xi}}
\ \int_0^T g_{\mu_k}(t,x(t),u(t))\,dt + \Psi(x(T))
\end{equation}
admits an optimal solution. Furthermore, assume there exists a $\rho_0$-measurable selection map $\xi \mapsto \omega_k(\xi)$ such that, for $\rho_0$-a.e. $\xi$, the trajectory $\omega_k(\xi) = (x_k(\cdot, \xi), u_k(\cdot, \xi)) \in \Gamma$ attains the minimum $V(\xi)$, where $\Gamma$ is the space of admissible trajectories defined in \eqref{eq:trajSpace}.

Define the aggregated measures:
\[
\bar\mu_k := \int_{\mathcal X} \mu[\omega_k(\xi)]\,d\rho_0(\xi),
\qquad
\bar\nu_k := \int_{\mathcal X} \nu[\omega_k(\xi)]\,d\rho_0(\xi).
\]
Then $(\bar\mu_k,\bar\nu_k)\in\Delta$ attains the minimum of the
FW subproblem:
\[
(\bar\mu_k,\bar\nu_k)\in
\argmin_{(\mu,\nu)\in\Delta}
\big( \langle g_{\mu_k},\mu\rangle+\langle\Psi,\nu\rangle \big).
\]
\end{Theorem}
While Corollary~\ref{Cor:exist} guarantees the existence of an optimal solution over $\Delta$, the minimum in~\eqref{eq:equivFWSub} over classical trajectories may fail to be attained in certain non-compact settings. In such cases, $\omega_k$ is understood as an $\varepsilon$-optimal solution of the subproblem. The convergence guarantees of FW methods remain valid under such inexact linear minimization oracles; see~\cite{yu2025deterministic}.

Crucially, this construction ensures that the algorithm never selects the ``ghost relaxations'' illustrated in Example~\ref{ex:chattering}. Because the LMO is restricted to the occupation measures induced by classical trajectories in $\Gamma$, the resulting FW iterates $\{\mu_k\}$ remain finite convex combinations of admissible dynamics satisfying $\dot{x}(t) = f(x(t), u(t))$. Consequently, as highlighted in Section~\ref{subsec:discussion}, our framework provides a sequence of finite mixtures of admissible trajectories whose costs approach the relaxed optimum $J^*$. These strategies are composed of discrete groups of agents following classical paths, thereby avoiding the non-physical averaged velocities inherent in the continuous modeling relaxation.

By Theorem~\ref{thm:lmo}, the FW subproblem on $\Delta$ reduces to solving a family of deterministic optimal control problems (OCPs) indexed by the initial states $\xi \sim \rho_0$. This reduction allows us to solve the infinite-dimensional measure optimization through a sequence of classical OCPs, as summarized in Algorithm~\ref{alg:FW}.

\begin{algorithm}
\caption{FW over Occupation Measures}
\label{alg:FW}
\begin{algorithmic}[1]
\Require Initial measures $(\mu_0,\nu_0)\in\Delta$ (consistent with $\rho_0$), step sizes $\{\alpha_k\}_{k\ge 0}$.
\For{$k=0,1,2,\dots, K$}
    \State Solve the parametric linearized subproblem: for $\rho_0$-a.e.\ $\xi\in\mathcal X$,
    \vspace{-8pt}
    \[
      \omega_k(\xi) \in \argmin_{\substack{x(0)=\xi\\ \dot x=f(x,u)}}
      \left\{ \int_0^T g_{\mu_k}(t,x(t),u(t))\,dt + \Psi(x(T)) \right\}.
    \]
    \vspace{-8pt}
    \State Construct the aggregated occupation measures
    \vspace{-8pt}
    \[
      \bar\mu_k = \int_{\mathcal X}\mu[\omega_k(\xi)]\,d\rho_0(\xi),\qquad
      \bar\nu_k = \int_{\mathcal X}\nu[\omega_k(\xi)]\,d\rho_0(\xi).
    \]
    \vspace{-8pt}
    \State
    $(\mu_{k+1},\nu_{k+1})
      \gets
      (1-\alpha_k)(\mu_k,\nu_k)
      + \alpha_k(\bar\mu_k,\bar\nu_k).$
\EndFor
\State \Return $(\mu_K,\nu_K)$.
\end{algorithmic}
\end{algorithm}
In practice, the spatial aggregation in Step~3 is implemented according to the physical structure of the initial distribution $\rho_0$. For a discrete distribution supported on $M$ locations $\{\xi^{(m)}\}_{m=1}^M$ with associated probabilities $\{\pi_m\}_{m=1}^M$, the aggregation $\bar{\mu}_k = \sum_{m=1}^M \pi_m \mu[\omega_k(\xi^{(m)})]$ is exact. For continuous distributions, the integral is approximated via Monte Carlo sampling, where $M$ samples drawn from $\rho_0$ yield the empirical approximation $\bar{\mu}_k \approx \frac{1}{M} \sum_{m=1}^M \mu[\omega_k(\xi^{(m)})]$. In both cases, the $M$ subproblems are mutually independent and can be solved in parallel, which is critical for maintaining computational efficiency in large-scale multi-agent settings.

One important remark is that although the optimal control problem in Step~2 of Algorithm~\ref{alg:FW} depends on the current iterate $\mu_k$, it remains tractable. Since $\mu_k$ is a convex combination of previously generated aggregated occupation measures, it can be written as $\mu_k=\sum_{i=0}^k \beta_i^{(k)} \bar\mu_i$ for some coefficients $\beta_i^{(k)}\ge 0$ satisfying $\sum_{i=0}^k \beta_i^{(k)}=1$. Consequently, the linearized cost $g_{\mu_k}$ admits a closed-form expression. For a given initial state $\xi$, the subproblem reduces to
\begin{align} \label{eq:LMO_final}
 &\min_{ \substack{x(0)=\xi, \\ \dot x=f(x,u), \\ u \in \mathcal{U}} } \quad \Psi(x(T)) + 
  \int_0^T \bigg( \ell_0(t,x,u) \nonumber\\
 & + 2\lambda \sum_{i=0}^k \beta_i^{(k)} \int_{\mathcal{X}} W(x(t) - x_i(t, \zeta)) \, d\rho_0(\zeta) \bigg) dt .
\end{align}
where $x_i(\cdot, \zeta)$ denotes the state trajectory component of $\omega_i(\zeta)$ generated at the $i$-th iteration. In practice, the spatial integral in \eqref{eq:LMO_final} is evaluated based on the structure of $\rho_0$. For a discrete or sampled distribution with $M$ initial states $\{\xi^{(m)}\}_{m=1}^M$ and weights $\{\pi_m\}_{m=1}^M$, the interaction term in \eqref{eq:LMO_final} simplifies to a double summation over the previously generated trajectory ensembles and the $M$ sample paths: $\sum_{i=0}^k \beta_i^{(k)} \sum_{m=1}^M \pi_m W(x(t) - x_i(t, \xi^{(m)}))$. This formulation yields a standard optimal control problem, for which a broad range of mature solution methods and solvers are available. In our numerical experiments, a gradient-based method is employed to illustrate that the FW subproblem can be addressed effectively in practice. Specifically, we utilize the adjoint method to compute gradients and perform updates on the control sequence $u$ using the Adam optimizer.


\subsection{Complexity of Frank--Wolfe over Occupation Measures}\label{subsec:fw-rate}
We next establish the convergence rate of Algorithm~\ref{alg:FW}. The result
extends the classical FW complexity bound for convex smooth
objectives on compact convex sets to the present optimization problem posed
over occupation measures.
\begin{Definition}[$L$-smoothness]\label{def:Lsmooth}
We say that $J$ is $L$-smooth on $\Delta$ with respect to the
product total variation norm if for all
$(\mu,\nu),(\mu',\nu')\in\Delta$,
\begin{align*}
J(\mu',\nu')
&\le
J(\mu,\nu)
+
\delta J(\mu,\nu;\mu'-\mu,\nu'-\nu)
\\
&\qquad
+
\frac{L}{2}
\|(\mu'-\mu,\nu'-\nu)\|^2 .
\end{align*}
\end{Definition}

\begin{Theorem}[FW Complexity]\label{thm:fw-rate}
Assume that $J$ is convex and $L$-smooth on $\Delta$.
Let $\{(\mu_k,\nu_k)\}$ be generated by Algorithm~\ref{alg:FW} with step sizes
$\alpha_k = \frac{2}{k+2}$.
Then for all $k\ge 1$,
\[
J(\mu_k,\nu_k) - J^*
\le \frac{8L(T+1)^2}{k+2},
\]
where $J^* := \min_{(\mu,\nu)\in\Delta} J(\mu,\nu)$.
\end{Theorem}
Here $T$ denotes the time horizon of the optimal control problem.
Since the occupation measure satisfies $\|\mu\| = T$ and the terminal measure satisfies $\|\nu\| = 1$, this yields the factor $T+1$ in the bound. The proof follows the standard FW descent argument adapted to the measure setting and is provided in Appendix~\ref{app:fw-proof}.

By Theorem~\ref{thm:conv}, the positive semidefiniteness of $W$ ensures that the objective $J$ is convex, and Theorem~\ref{thm:fw-rate} therefore yields the standard $\mathcal{O}(1/k)$ convergence rate. If $W$ is not positive semidefinite, the objective $J$ may become nonconvex. In this case, FW methods still guarantee convergence to a stationary point characterized by the variational inequality~\eqref{eq:VI}; see~\cite{jaggi2013revisiting,yu2025deterministic}.

\subsection{Fully-Corrective Frank--Wolfe}\label{subsec:fcfw}
A natural extension of Algorithm~\ref{alg:FW} is the
FCFW variant, which often yields faster practical convergence \cite{yu2025derivative}. Instead of forming a simple convex combination between the current iterate and the aggregated occupation--measure pair induced by the newly generated trajectory ensemble, the FCFW variant re-optimizes the weights over all previously generated aggregated occupation measures associated with admissible trajectories. At iteration $k$, the fully-corrective update replaces Step~4 of
Algorithm~\ref{alg:FW} by solving
\begin{equation}\label{eq:fcfw}
\min_{\alpha \in \mathbb{R}^{k+1}}
J\!\left(\sum_{i=0}^k \alpha_i (\bar{\mu}_i,\bar{\nu}_i)\right)
\quad
\text{s.t. } \alpha \ge 0, \ \mathbf{1}^\top \alpha = 1.
\end{equation}
Due to the bilinear structure of the interaction term in
\eqref{mainObj}, problem \eqref{eq:fcfw} reduces to a quadratic program
(QP)
\begin{equation}\label{eq:fcfw-qp}
\min_{\alpha \ge 0,\, \mathbf{1}^\top \alpha = 1}
\frac{1}{2}\alpha^\top H \alpha + c^\top \alpha .
\end{equation} The linear cost vector $c \in \mathbb{R}^{k+1}$ has components
\begin{equation*}
 c_i := \int_{\mathcal{X}} \bigg( \int_0^T \ell_0(t, x_i, u_i) \, dt 
 + \Psi(x_i(T)) \bigg) d\rho_0
\end{equation*}
while the interaction matrix $H \in \mathbb{R}^{(k+1)\times(k+1)}$ is
\begin{equation*}
H_{ij}
:=
2\lambda
\int_0^T
\iint_{\mathcal X^2}
W(x_i(t,\zeta)-x_j(t,\eta))
\, d\rho_0(\zeta)\, d\rho_0(\eta)
\, dt .
\end{equation*} If the interaction kernel $W$ is positive semidefinite, the matrix $H$
is likewise positive semidefinite, and the problem
\eqref{eq:fcfw-qp} is a convex QP. Under the assumptions of
Theorem~\ref{thm:fw-rate}, the fully-corrective variant admits the same
$\mathcal{O}(1/k)$ convergence rate as Algorithm~\ref{alg:FW}. The resulting fully-corrective procedure is summarized in Algorithm~\ref{alg:FCFW}.

\begin{algorithm}
\caption{FCFW over Occupation Measures}
\label{alg:FCFW}
\begin{algorithmic}[1]
\Require Initial measures $(\mu_0,\nu_0)\in\Delta$ (consistent with $\rho_0$). Set $(\bar\mu_0,\bar\nu_0):=(\mu_0,\nu_0)$
\For{$k=0,1,2,\dots,K-1$}
    \State Solve the parametric linearized subproblem: for $\rho_0$-a.e.\ $\xi\in\mathcal X$,
    \vspace{-8pt}
    \[
      \omega_k(\xi)\in
      \argmin_{\substack{x(0)=\xi\\ \dot x=f(x,u)}}
      \left\{
      \int_0^T g_{\mu_k}(t,x(t),u(t))\,dt+\Psi(x(T))
      \right\}.
    \]
    \vspace{-8pt}
    \State Construct the aggregated occupation measures
    \vspace{-8pt}
    \[
      \bar\mu_{k+1} = \int_{\mathcal X}\mu[\omega_k(\xi)]\,d\rho_0(\xi),\quad
      \bar\nu_{k+1} = \int_{\mathcal X}\nu[\omega_k(\xi)]\,d\rho_0(\xi).
    \]
    \vspace{-8pt}
    \State Solve the fully-corrective problem
    \vspace{-8pt}
    \[
      \alpha^{(k+1)}\in
      \argmin_{\alpha\ge0,\ \mathbf 1^\top\alpha=1}
      J\!\left(\sum_{i=0}^{k+1}\alpha_i(\bar\mu_i,\bar\nu_i)\right).
    \]
    \vspace{-8pt}
    \State Update
    $
      (\mu_{k+1},\nu_{k+1})
      \gets
      \sum_{i=0}^{k+1}\alpha_i^{(k+1)}(\bar\mu_i,\bar\nu_i).
    $
\EndFor
\State \Return $(\mu_K,\nu_K)$.
\end{algorithmic}
\end{algorithm}

In our numerical implementation, the vector $c$ and matrix $H$ are
constructed using the particle representation of $\rho_0$ described
earlier, with $M$ representative samples $\{\xi^{(m)}\}_{m=1}^M$. Forming $H$ involves $O(k^2 M^2)$ pairwise interactions in total, but the matrix can be updated incrementally by computing only the interactions between the new trajectory ensemble $\omega_k$ and the previous ones $\{\omega_0,\dots,\omega_{k-1}\}$. Since the dimension of the resulting QP \eqref{eq:fcfw-qp} is only $k+1$, solving it is computationally inexpensive compared with the optimal control subproblem in Step~2. In the numerical experiments we therefore adopt the fully-corrective variant to improve practical convergence speed.

\section{Applications and numerical results}\label{simulation}
This section presents numerical studies of the proposed OM-MFC and FCFW framework on two representative applications. The first considers UAV swarm coordination with obstacle avoidance, and the second considers satellite trajectory planning under three-dimensional Keplerian dynamics. Together, these examples illustrate the flexibility and scalability of the proposed method across different dynamical models and application settings.
\subsection{UAV Swarm Coordination and Obstacle Avoidance}

We consider a swarm of UAVs operating in environments with static obstacles. The objective is to steer the swarm from an initial spatial distribution toward a target configuration while minimizing a cost that penalizes control effort, discourages high-density clustering, and enforces convergence to the desired terminal region. This setting illustrates how the proposed occupation-measure framework handles state constraints and population-level coupling in large multi-agent systems.

The numerical experiments are organized into three representative settings that highlight different aspects of the proposed method.

\noindent\textbf{(i) 2D swarm with a single initial location.}
In this case, a baseline solution can be obtained from a Hamilton--Jacobi--Bellman and Fokker--Planck (HJB--FP) PDE formulation. The trajectories and distributions produced by the FW method closely match the PDE solution, confirming the consistency of the proposed framework.
\noindent\textbf{(ii) 2D swarm with multiple initial locations.}
We next consider a scenario where the swarm originates from several distinct starting locations. This setting represents practical situations in which UAVs are deployed from multiple bases. A PDE solution can still be computed in this two-dimensional case, and the swarm evolution obtained from the FW method remains qualitatively consistent with the PDE reference while demonstrating that the proposed framework naturally accommodates distributed initial conditions.
\noindent\textbf{(iii) 3D swarm with multiple obstacles.}
Finally, we extend the problem to a three-dimensional environment with multiple obstacles. The FW algorithm generates smooth, collision-free trajectories while guiding the swarm toward the target region, illustrating the scalability of the proposed framework in higher-dimensional environments.

\subsubsection{Problem setup}

We apply the proposed framework to a UAV swarm tasked with moving from an initial spatial distribution toward a target region centered at \(x_g\) over a finite horizon \(T\). The agent dynamics are modeled as single integrators:
\[
\dot x(t)=u(t).
\]
The corresponding occupation-measure objective takes the form
\begin{align}\label{eq:sim_measure}
J &= \int \Big(\frac{\alpha}{2}\|u\|^2 + V_{\mathrm{obs}}(x)\Big)\,d\mu\nonumber\\
&+ \frac{\lambda_\Psi}{2}\int_{\mathcal{X}}\|x-x_g\|^2\,d\nu(x)\\
& + \frac{\gamma}{2}\iint
W(\|x-y\|)\,\delta(t-t')\,d\mu(t,x,u)\,d\mu(t',y,v)\nonumber,
\end{align}
The interaction potential $W$ induces coupling across the population, so the objective cannot be decomposed into independent per-agent terms. The repulsive interaction is modeled by an isotropic Gaussian kernel
\(
W(r) = \exp\!\{-r^{2}/(2\sigma^{2})\}
\),
with width $\sigma$ corresponding to the sensing radius, and $\gamma$ tuning the overall repulsion strength.
Static obstacles are penalized by
\begin{equation}
V_{\mathrm{obs}}(x) = \beta\,\max\!\bigl(0,\, R - \|x - c\|\bigr)^2,
\end{equation}
where $c$ and $R$ are the obstacle center and radius respectively, and $\beta > 0$ is a penalty gain.

\subsubsection{2D Swarm with Repulsion and a Single Obstacle (Effectiveness)}\label{2D single}
We consider a two-dimensional swarm coordination problem in which the agents must maintain separation while navigating around a single circular obstacle of radius $r_{\mathrm{obs}}=0.8$ centered at $(2.5,\,1.5)$, together with a safety margin of $0.2$. Obstacle avoidance is enforced through a quadratic penalty potential with weight $\beta=10^3$, and inter-agent repulsion is introduced via the Gaussian kernel $W$ with parameters $\lambda_W=0.5$ and $\sigma_W=0.25$. The remaining parameters are chosen as follows: horizon $T=3$, control weight $\alpha=0.1$, terminal weight $\lambda_\Psi=30$, target position $x_g=(5,3)$, and a point-mass initial condition at $(0,0)$.

We implement the FCFW algorithm (Algorithm~\ref{alg:FCFW}) with $K=100$ outer iterations. At each iteration, the linear minimization oracle requires solving an optimal control subproblem, which in our setting reduces to a QP with linear dynamics and quadratic cost. This subproblem is solved using the Adam optimizer with $400$ steps and learning rate $0.2$. The resulting trajectory is added to the current dictionary, after which the mixture weights over all accumulated trajectories are reoptimized by solving \eqref{eq:fcfw} via projected gradient descent on the simplex.
\begin{figure}[t]
    \centering
    \includegraphics[width=\linewidth]{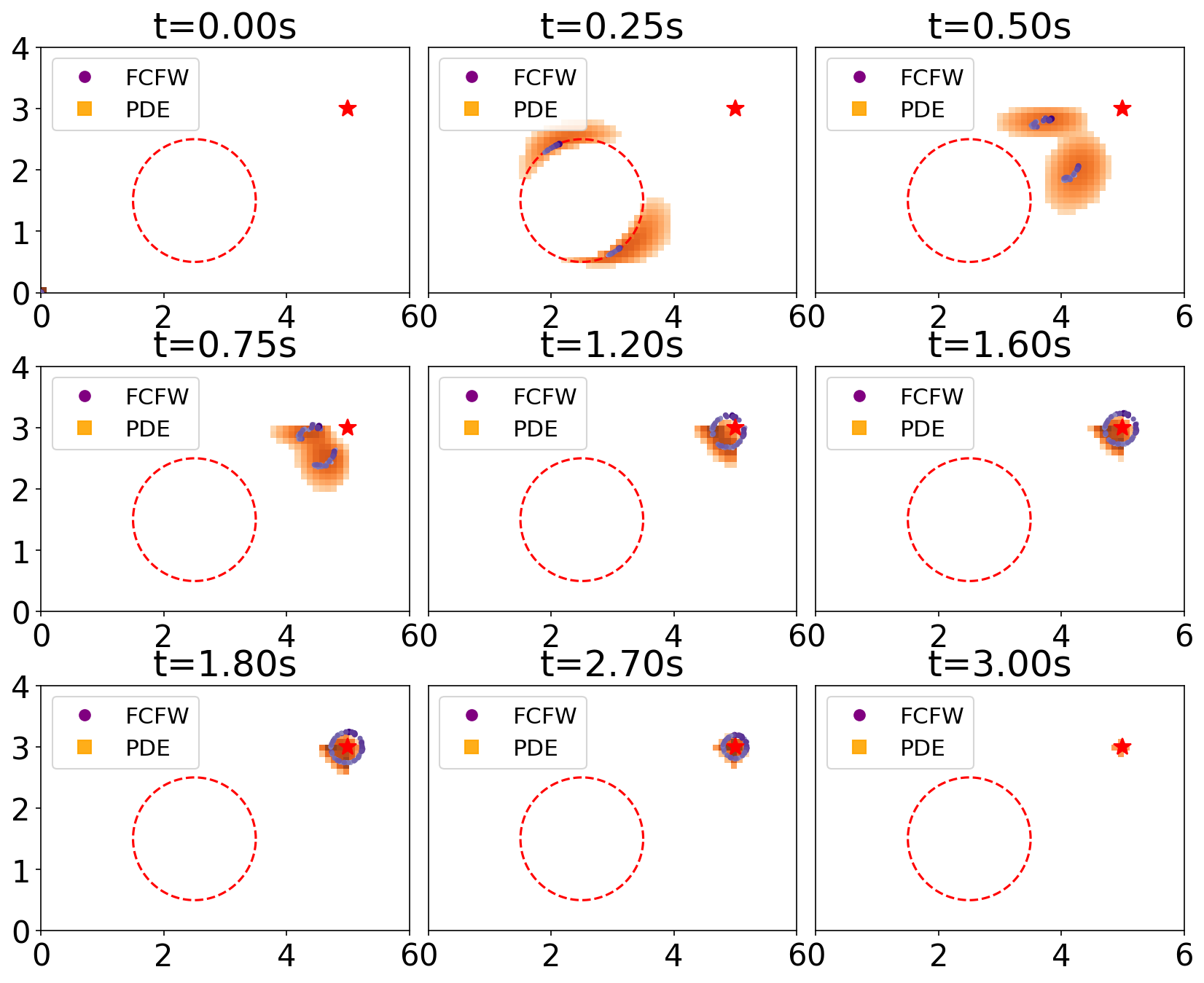}
    \caption{2D swarm with one obstacle (point-mass initial state): FCFW (purple points) vs. PDE (orange heatmap).}
    \label{fig:obstacle_compare_2d}
\end{figure}
\begin{figure}[t]
    \centering
    \includegraphics[width=1.0\linewidth]{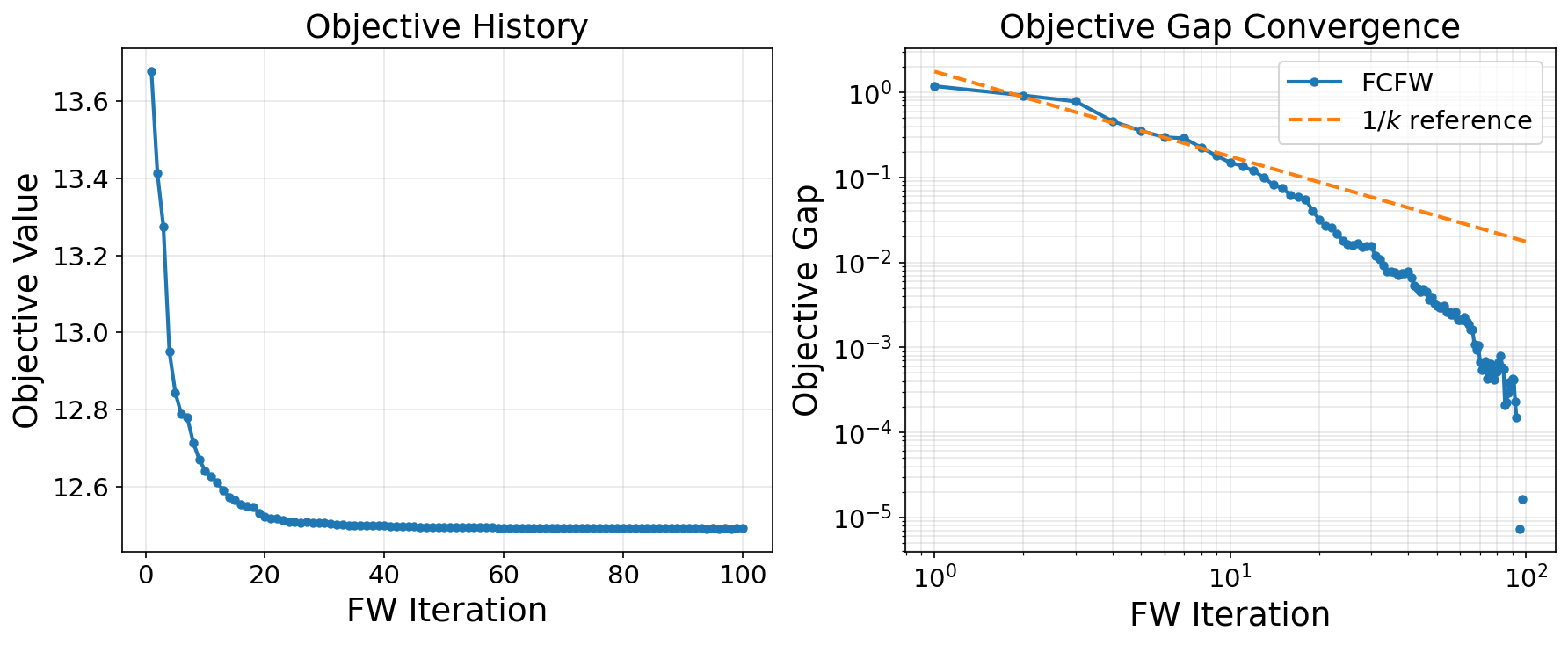}
    \caption{Objective and objective-gap convergence (point-mass initial state).}
    \label{fig:obj_converge}
\end{figure}

Figure~\ref{fig:obstacle_compare_2d} compares the FCFW solution with the PDE benchmark in the presence of a repulsive potential and a single circular obstacle. As discussed in Section~\ref{relation}, the OM-MFC formulation is closely related to the classical MFC PDE system at the population level, but the two approaches differ in their modeling structure and computational representation: the PDE benchmark is obtained from grid-based density evolution, whereas the proposed method is computed through trajectory-based optimization in the FCFW framework. For this reason, we do not compare the corresponding objective values directly, and instead use the PDE solution as a continuum reference for the swarm evolution. In this example, both methods exhibit qualitatively consistent behavior: starting from a single point in the lower-left corner, the swarm splits around the obstacle, disperses under the repulsive interaction, and then reconverges toward the target at $(5,3)$ from multiple directions. This agreement indicates that the proposed trajectory-based method captures the same population-level behavior in this low-dimensional setting.
Figure~\ref{fig:obj_converge} shows the objective history under the FCFW algorithm together with the decay of the objective gap on a log--log scale. Here the objective gap is defined as $J_k - J_K$
where $J_k$ denotes the objective value at iteration $k$, and $J_K$ is the final objective value. The objective decreases steadily over the iterations and quickly approaches a stable value. In the right panel, the displayed gap exhibits a sublinear decay broadly consistent with the \(1/k\) reference slope, in agreement with the theoretical convergence behavior of the FW method. These results indicate stable convergence of the proposed algorithm in this obstacle-avoidance setting.

\subsubsection{2D Swarm with Distributed Initial States ((General Initial Distributions))}\label{2D multiple}

A practically important scenario in UAV swarm coordination arises when agents are deployed from multiple geographically separated bases rather than a single launch point. This setting captures realistic operations such as coordinated area coverage, search-and-rescue missions, or joint strike packages, where UAVs are launched from distinct forward operating locations and must coordinate to accomplish a common task. 

From the algorithmic perspective, the OM-MFC framework accommodates this setting naturally
through the initial distribution $\rho_0$. Rather than restricting $\rho_0$ to a point mass
as in Section~\ref{2D single}, we now take $\rho_0$ to be a discrete measure supported on $M$ distinct
initial locations $\{\xi^{(m)}\}_{m=1}^{M}$ with uniform weights $\pi_m = 1/M$. By
Theorem~\ref{thm:lmo}, the FW subproblem at each iteration decomposes into $M$ independent
optimal control problems, one for each starting location $\xi^{(m)}$. The aggregated occupation measure is then formed as the weighted empirical average
\begin{equation}
    \bar{\mu}_k = \frac{1}{M}\sum_{m=1}^{M} \mu[\omega_k(\xi^{(m)})],
\end{equation}
and the fully-corrective weight re-optimization in~\eqref{eq:fcfw-qp} proceeds over the accumulated dictionary of such aggregated measures. This preserves the same algorithmic structure as in the single-source case, while introducing $M$ mutually independent subproblems that can be solved in parallel.

We consider a swarm originating from $M = 10$ distinct initial positions, reflecting a realistic multi-base deployment scenario.
All other parameters are identical to Section~\ref{2D single}: target position $x_g = (5, 3)$,
horizon $T = 3$, a single circular obstacle of radius $r_{\mathrm{obs}} = 0.8$ centered
at $(2.5, 1.5)$ with safety margin $0.2$, obstacle penalty weight $\beta = 10^3$,
inter-agent repulsion parameters $\lambda_W = 0.5$ and $\sigma_W = 0.25$, control weight
$\alpha = 0.1$, and terminal weight $\lambda_\Psi = 30$. The FCFW algorithm is run for
$K = 100$ outer iterations, with the LMO subproblem solved via the Adam optimizer with
$400$ gradient steps and learning rate $0.2$.
\begin{figure}[t]
    \centering
    \includegraphics[width=1.05\linewidth]{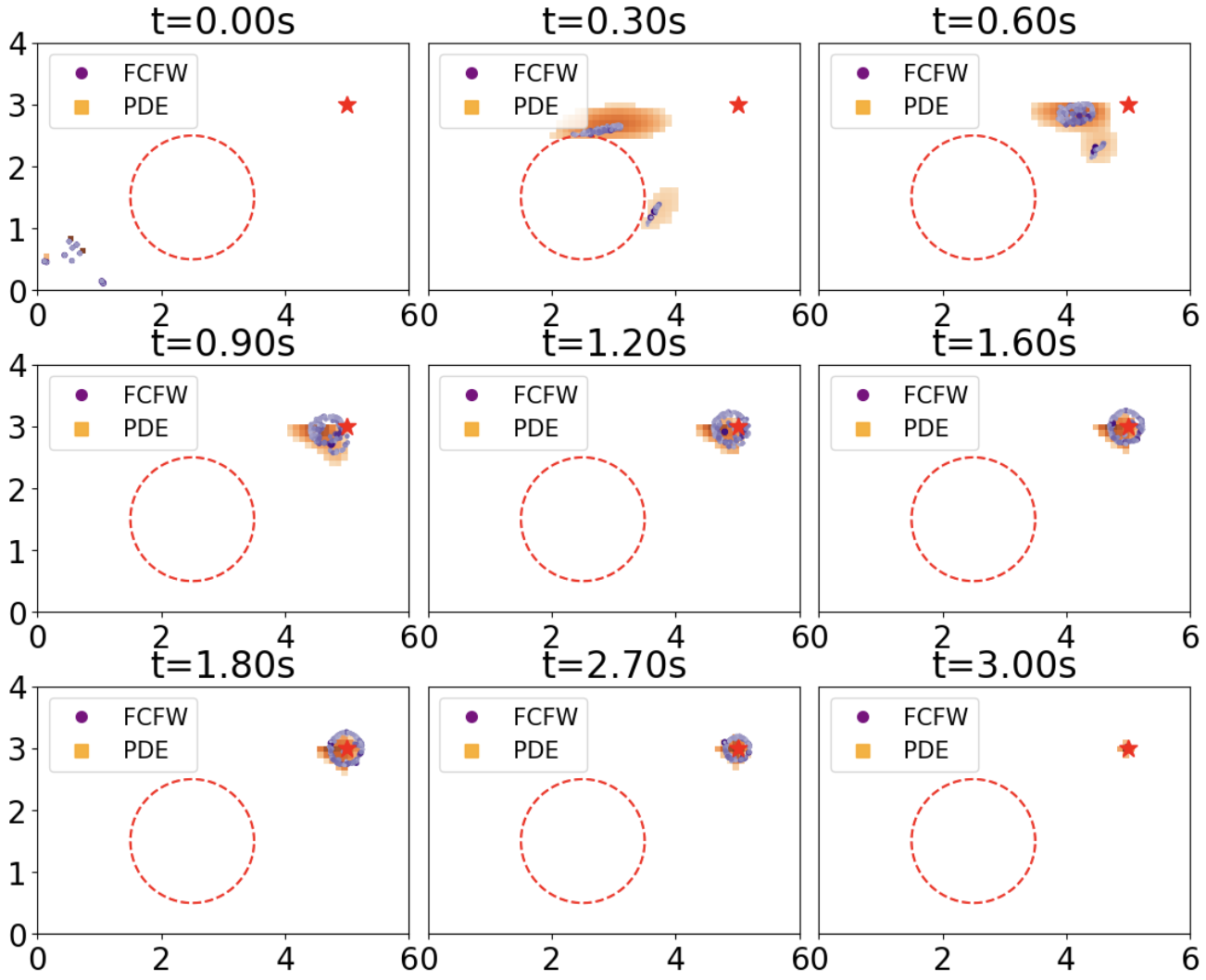}
    \caption{2D swarm with one obstacle (distributed initial states): FCFW (purple points) vs. PDE (orange heatmap).}
    \label{fig:multi_source_snapshots}
\end{figure}
\begin{figure}[t]
    \centering
    \includegraphics[width=1.0\linewidth]{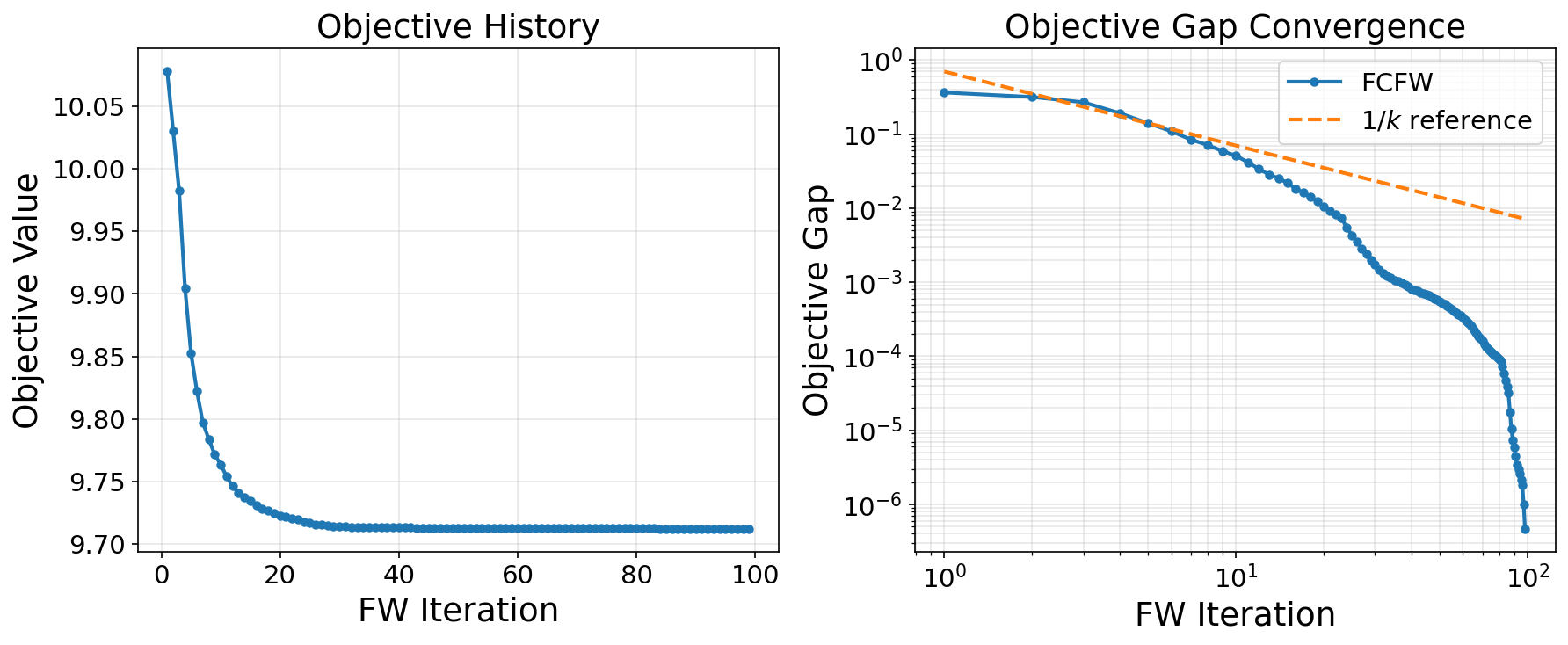}
    \caption{Objective and objective-gap convergence (distributed initial state).}
    \label{fig:multi_source_convergence}
\end{figure}

Figure~\ref{fig:multi_source_snapshots} shows snapshots of the swarm evolution at
representative time instants, comparing the FCFW solution (blue points) against the PDE
benchmark (orange heatmap). Despite the agents starting from spatially dispersed locations,
the FCFW trajectories remain in close qualitative agreement with the PDE density throughout
the entire horizon. Unlike the previous example, in this case a larger portion of the agents passes the obstacle from one side, while a smaller portion passes from the other. The repulsive interaction again promotes dispersion during transit, whereas the terminal constraint drives the population to reconverge near the target region at $t = T$.  As in Section~\ref{2D single}, the comparison is intended as a qualitative benchmark against the classical PDE formulation.
Figure~\ref{fig:multi_source_convergence} reports the objective history and the approximate objective gap decay on a log--log scale. The objective decreases steadily across iterations and stabilizes reliably, indicating that distributed initialization does not adversely affect the convergence behavior. The displayed gap exhibits a sublinear decay broadly consistent with the $\mathcal{O}(1/k)$ reference slope predicted by Theorem~\ref{thm:fw-rate}. These results demonstrate the robustness of the FCFW framework with respect to the initial distribution and confirm that the multi-source decomposition introduced in Section~\ref{sec:fw_methods} produces high-quality solutions without modifying the core algorithm.

The same construction also extends naturally to continuous initial distributions by replacing $\rho_0$ with a Monte Carlo empirical approximation, while preserving the decomposable structure of the subproblem.

\subsubsection{3D Swarm with Multiple Obstacles (Scalability)}

We next consider a three-dimensional swarm coordination problem with ten obstacles and repulsive interactions. Unlike the two-dimensional experiments in Sections~\ref{2D single} and~\ref{2D multiple}, we do not include a PDE benchmark here. In three dimensions, PDE-based approaches typically require substantially larger state-space discretizations, which makes such computations significantly more demanding in our current experimental setting. Our goal in this example is therefore to demonstrate the convergence and obstacle-avoidance behavior of the proposed FW framework in a higher-dimensional setting without relying on large-scale grid discretization. The parameters are chosen as follows: $x_0=(0,0,0),\quad x_g=(5,3,2),\quad T=3,\; N_t=150,\;\alpha=0.1,\; \lambda_\Psi=30,\; \lambda_W=25,\; \beta=10^3, \;K=50$.
\begin{figure}[t] 
\centering 
\includegraphics[width=\linewidth]{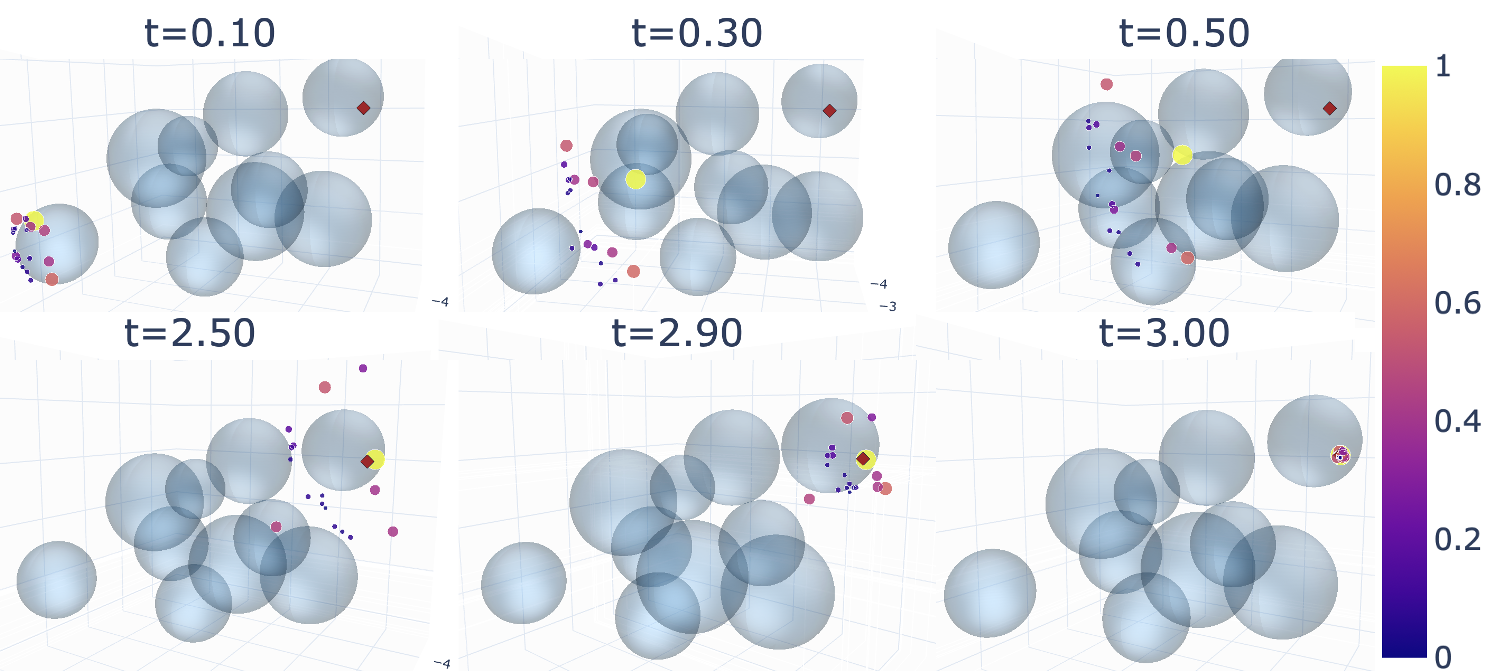} 
\caption{3D swarm with repulsion and 10 obstacle.} 
\label{fig:obstacle_compare_3d} 
\end{figure}
\begin{figure}[t] 
\centering 
\includegraphics[width=\linewidth]{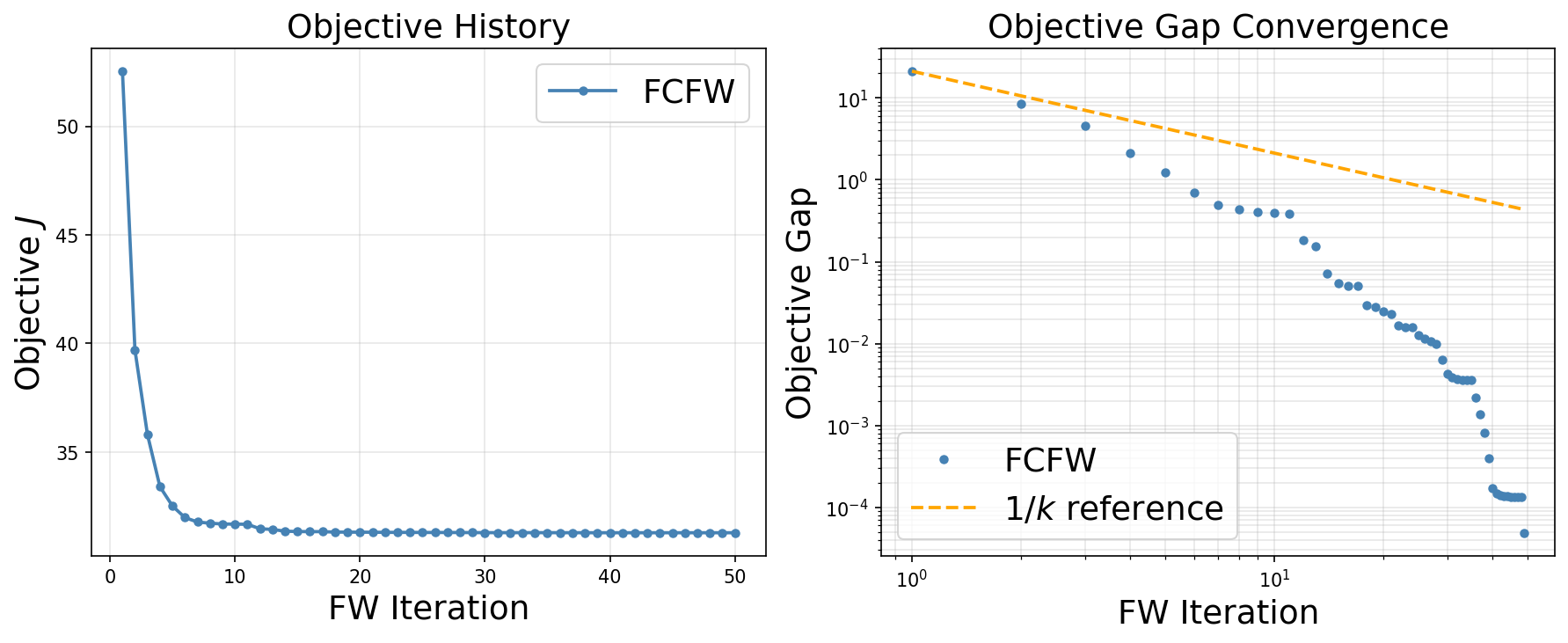} 
\caption{Objective and objective-gap convergence (3D).} 
\label{fig:fw-3d-obj} 
\end{figure}

Figure~\ref{fig:obstacle_compare_3d} shows time slices of the UAV distribution obtained from the FW solution.
The swarm successfully avoids obstacles, maintains repulsive separation, and gradually converges to the target region. 
In addition, Figure~\ref{fig:fw-3d-obj} reports the objective convergence history, confirming that the outer iterations stabilize rapidly. 
The entire 3D instance, including interactions and ten obstacles, is solved in about 20 minutes on a standard workstation, demonstrating the practicality of the proposed approach in higher-dimensional settings where grid-based discretization becomes prohibitively expensive.

\subsection{Satellite Constellation Coordination}
We next demonstrate the proposed framework on a second application: satellite constellation deployment under three-dimensional Keplerian orbital dynamics. Each satellite evolves according to nonlinear Keplerian dynamics in three-dimensional space, with a six-dimensional state consisting of position and velocity, making the problem significantly more complex than the two-dimensional single-integrator model considered previously.

The scenario is motivated by a rideshare deployment, in which multiple satellites are deployed from a common release point on a circular low Earth orbit (LEO) at radius $R_0 = 7000$ km. The satellites must transfer to a higher target orbit at $R_{\mathrm{target}} = 8000$ km within a horizon of $T = 6000$ s, while simultaneously spreading along the target orbit. Fuel consumption is penalized throughout the maneuver. This setting captures the core challenge of constellation deployment, where satellites must coordinate their final positions along the target orbit without explicit assignment---a situation that arises naturally in the commercial deployment of large satellite constellations used for global broadband communications or Earth observation.
\subsubsection{Problem Setup}

Each satellite's dynamics follow Keplerian orbital mechanics with thrust acceleration as control input:
\begin{align}
    \dot{r} = v, \;\;\dot{v} = -\frac{\mu}{\|r\|^3} r + a,
\end{align}
where $r \in \mathbb{R}^3$ and $v \in \mathbb{R}^3$ are the position and velocity, $\mu = 398600$ km$^3$/s$^2$ is the Earth's gravitational parameter, and $a \in \mathbb{R}^3$ is the thrust acceleration subject to the constraint $\|a\| \leq a_{\max} = 10^{-3}$ km/s$^2$.

The objective functional takes the form
\begin{align}
    J(\mu) = &\int_0^T \alpha \|a(t)\|^2 \, dt + \beta_f \bigl(\|r(T)\| - R_{\text{target}}\bigr)^2 \\
    &+ \lambda_W \int_0^T \int W_\sigma(r - r') \, d\mu(r') \, dt,
\end{align}
where the first term penalizes fuel consumption with weight $\alpha = 10^4$, the second term enforces arrival at the target orbit with weight $\beta_f = 10.0$, and the third term is a pairwise Gaussian interaction kernel $W_\sigma(r) = \exp\!\left(-\frac{\|r\|^2}{2\sigma^2}\right)$ with bandwidth $\sigma = 50$ km, weighted by $\lambda_W = 1000$, that promotes spatial dispersion among satellites along the target orbit.

\subsubsection{Results and Discussion}

We apply the FW algorithm over occupation measures to the constellation coordination problem, running $K = 100$ iterations with $500$ gradient steps per LMO subproblem.

\begin{figure}[h]
    \centering
    \includegraphics[width=0.45\textwidth]{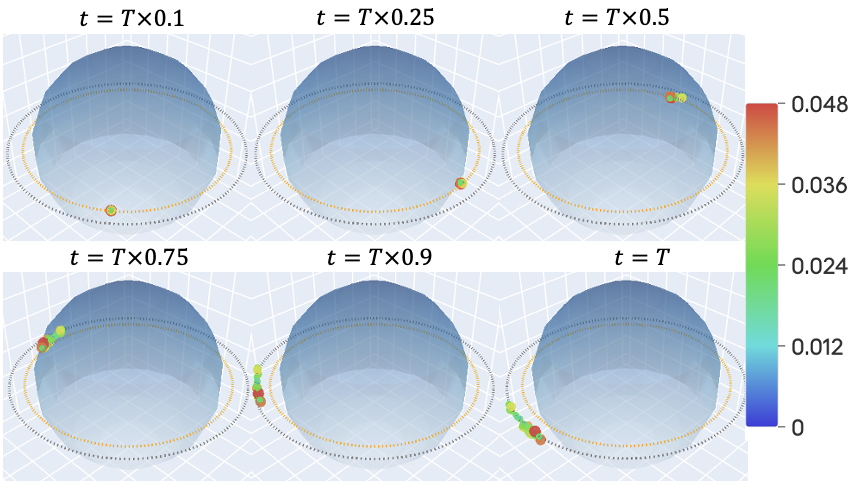}
    \caption{Snapshots of the weighted satellite distribution.}
    \label{fig:constellation_snapshots}
\end{figure}

\begin{figure}[h]
    \centering
    \includegraphics[width=0.5\textwidth]{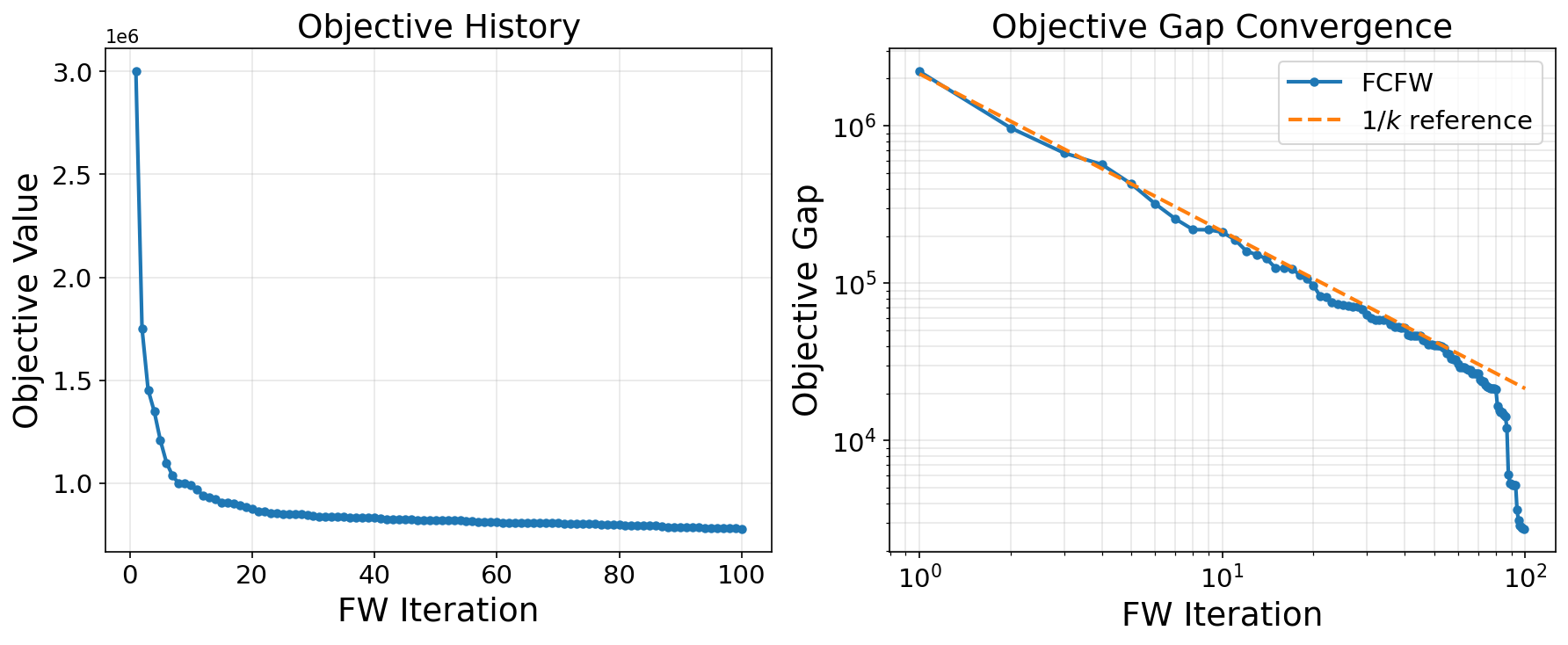}
    \caption{Objective and objective-gap convergence (satellite constellation coordination).}
    \label{fig:constellation_gap}
\end{figure}

Figure~\ref{fig:constellation_snapshots} shows snapshots of the weighted satellite distribution at six time instances. All satellites depart from the same initial position on the LEO at $t = 0$, and progressively spread along the target orbit as $t \to T$. Importantly, no in-plane constraint is imposed: the satellites are free to maneuver in full three-dimensional space, and the resulting active trajectories exhibit distinct orbital inclinations.
Figure~\ref{fig:constellation_gap} shows the objective value history and the FW objective gap on a log-log scale. The objective decreases rapidly in the first few iterations and then continues to improve more gradually. The objective gap closely tracks the theoretical $\mathcal{O}(1/k)$ convergence rate indicated by the reference line.

\section{Conclusions}\label{conclusion}
This paper proposed OM-MFC, which formulates large-population control directly as an optimization over occupation measures subject to linear dynamical constraints, and established convexity, weak compactness, and existence of optimal solutions under a positive-semidefinite interaction kernel. We developed FW and FCFW algorithms tailored to this setting, showing that each iteration reduces to solving classical optimal control subproblems and that the resulting iterates are finite mixtures of admissible trajectories with $\mathcal{O}(1/k)$ convergence guarantees. Numerical experiments on UAV swarm coordination and satellite constellation control demonstrated that the framework matches PDE-based baselines in low dimensions and scales to three-dimensional constrained environments at practical computational cost. An interesting direction for future work is to investigate how the population-level solutions produced by OM-MFC can be deployed in systems with a finite number of agents, for example by assigning agents to representative trajectories according to their associated weights. The proposed framework thus provides a promising planning layer for large-scale autonomous systems, and future work will also explore its integration with receding-horizon implementations, stochastic models, and distributed realizations.
\bibliographystyle{plain}        
\bibliography{autosam}           

\begin{thebibliography}{10}

\bibitem{ambrosio2005gradient}
Luigi Ambrosio, Nicola Gigli, and Giuseppe Savar{\'e}.
\newblock {\em Gradient flows: in metric spaces and in the space of probability measures}.
\newblock Springer, 2005.

\bibitem{bogachev2007measure}
Vladimir~I Bogachev.
\newblock {\em Measure theory}.
\newblock Springer, 2007.

\bibitem{bonnans2023large}
J~Fr{\'e}d{\'e}ric Bonnans, Kang Liu, Nadia Oudjane, Laurent Pfeiffer, and Cheng Wan.
\newblock Large-scale nonconvex optimization: randomization, gap estimation, and numerical resolution.
\newblock {\em SIAM Journal on Optimization}, 33(4):3083--3113, 2023.

\bibitem{2017boygeorec}
Nicholas Boyd, Geoffrey Schiebinger, and Benjamin Recht.
\newblock The alternating descent conditional gradient method for sparse inverse problems.
\newblock {\em SIAM Journal on Optimization}, 27(2), 2017.

\bibitem{2015bub}
S.~Bubeck.
\newblock Convex optimization: Algorithms and complexity.
\newblock {\em Foundations and Trends in Machine Learning}, 8(3--4):231--358, 2015.

\bibitem{cardaliaguet2017learning}
Pierre Cardaliaguet and Saeed Hadikhanloo.
\newblock Learning in mean field games: the fictitious play.
\newblock {\em ESAIM: Control, Optimisation and Calculus of Variations}, 23(2):569--591, 2017.

\bibitem{carmona2018probabilistic}
Ren{\'e} Carmona, Fran{\c{c}}ois Delarue, et~al.
\newblock {\em Probabilistic theory of mean field games with applications I-II}, volume~3.
\newblock Springer, 2018.

\bibitem{chen2020mean}
Dezhi Chen, Qi~Qi, Zirui Zhuang, Jingyu Wang, Jianxin Liao, and Zhu Han.
\newblock Mean field deep reinforcement learning for fair and efficient uav control.
\newblock {\em IEEE Internet of Things Journal}, 8(2):813--828, 2020.

\bibitem{du2025self}
Kai Du, Zhenjie Ren, Florin Suciu, and Songbo Wang.
\newblock Self-interacting approximation to mckean--vlasov long-time limit: a markov chain monte carlo method.
\newblock {\em Journal de Math{\'e}matiques Pures et Appliqu{\'e}es}, page 103782, 2025.

\bibitem{elamvazhuthi2019mean}
Karthik Elamvazhuthi and Spring Berman.
\newblock Mean-field models in swarm robotics: A survey.
\newblock {\em Bioinspiration \& Biomimetics}, 15(1):015001, 2019.

\bibitem{emerick2023continuum}
Max Emerick and Bassam Bamieh.
\newblock Continuum swarm tracking control: A geometric perspective in wasserstein space.
\newblock In {\em 2023 62nd IEEE Conference on Decision and Control (CDC)}, pages 1367--1374. IEEE, 2023.

\bibitem{fattorini1999infinite}
Hector~O Fattorini.
\newblock {\em Infinite dimensional optimization and control theory}, volume~54.
\newblock Cambridge University Press, 1999.

\bibitem{filippov2013differential}
Aleksei~Fedorovich Filippov.
\newblock {\em Differential equations with discontinuous righthand sides: control systems}.
\newblock Springer Science \& Business Media, 2013.

\bibitem{fornasier2014mean}
Massimo Fornasier and Francesco Solombrino.
\newblock Mean-field optimal control.
\newblock {\em ESAIM: Control, Optimisation and Calculus of Variations}, 20(4):1123--1152, 2014.

\bibitem{fung2025mean}
Samy~Wu Fung and Levon Nurbekyan.
\newblock Mean-field control barrier functions: A framework for real-time swarm control.
\newblock In {\em 2025 American Control Conference (ACC)}, pages 1--7. IEEE, 2025.

\bibitem{hu2025mf}
Anran Hu and Junzi Zhang.
\newblock Mf-oml: Online mean-field reinforcement learning with occupation measures for large population games.
\newblock {\em Applied Mathematics \& Optimization}, 92(3):1--54, 2025.

\bibitem{jaggi2013revisiting}
Martin Jaggi.
\newblock Revisiting frank-wolfe: Projection-free sparse convex optimization.
\newblock In {\em International conference on machine learning}, pages 427--435. PMLR, 2013.

\bibitem{1960kie}
Jack Kiefer.
\newblock Optimum experimental designs {V}, with applications to systematic and rotatable designs.
\newblock In {\em Proceedings of the fourth Berkeley symposium on mathematical statistics and probability}, volume~1, pages 381--405. Univ of California Press, 1960.

\bibitem{konig2017mean}
Wolfgang K{\"o}nig and Chiranjib Mukherjee.
\newblock Mean-field interaction of brownian occupation measures, i: Uniform tube property of the coulomb functional.
\newblock 2017.

\bibitem{lacker2017limit}
Daniel Lacker.
\newblock Limit theory for controlled mckean--vlasov dynamics.
\newblock {\em SIAM Journal on Control and Optimization}, 55(3):1641--1672, 2017.

\bibitem{lasserre2008nonlinear}
Jean~B Lasserre, Didier Henrion, Christophe Prieur, and Emmanuel Tr{\'e}lat.
\newblock Nonlinear optimal control via occupation measures and lmi-relaxations.
\newblock {\em SIAM journal on control and optimization}, 47(4):1643--1666, 2008.

\bibitem{lavigne2023generalized}
Pierre Lavigne and Laurent Pfeiffer.
\newblock Generalized conditional gradient and learning in potential mean field games.
\newblock {\em Applied Mathematics \& Optimization}, 88(3):89, 2023.

\bibitem{2018mei}
Song Mei, Andrea Montanari, and Phan-Minh Nguyen.
\newblock A mean field view of the landscape of two-layer neural networks.
\newblock {\em Proceedings of the National Academy of Sciences}, 115(33), jul 2018.

\bibitem{ruthotto2020machine}
Lars Ruthotto, Stanley~J Osher, Wuchen Li, Levon Nurbekyan, and Samy~Wu Fung.
\newblock A machine learning framework for solving high-dimensional mean field game and mean field control problems.
\newblock {\em Proceedings of the National Academy of Sciences}, 117(17):9183--9193, 2020.

\bibitem{vidal2025kernel}
Alexander Vidal, Samy~Wu Fung, Stanley Osher, Luis Tenorio, and Levon Nurbekyan.
\newblock Kernel expansions for high-dimensional mean-field control with non-local interactions.
\newblock In {\em 2025 American Control Conference (ACC)}, pages 4164--4171. IEEE, 2025.

\bibitem{vinter2010optimal}
Richard~B Vinter and RB~Vinter.
\newblock {\em Optimal control}, volume~2.
\newblock Springer, 2010.

\bibitem{yu2025derivative}
Di~Yu, Shane~G Henderson, and Raghu Pasupathy.
\newblock The derivative-free fully-corrective frank-wolfe algorithm for optimizing functionals over probability spaces.
\newblock In {\em 2025 Winter Simulation Conference (WSC)}, pages 3358--3369. IEEE, 2025.

\bibitem{yu2025deterministic}
Di~Yu, Shane~G. Henderson, and Raghu Pasupathy.
\newblock Deterministic and stochastic frank-wolfe recursion on probability spaces.
\newblock {\em Mathematics of Operations Research}, 2025.

\bibitem{yu2025frank}
Di~Yu, Shane~G Henderson, and Raghu Pasupathy.
\newblock Frank-wolfe recursions for the emergency response problem on measure spaces.
\newblock {\em arXiv preprint arXiv:2507.09808}, 2025.

\bibitem{yuyoupei2026ACC}
Di~Yu, Sixiong You, and Chaoying Pei.
\newblock Convexifying mean-field control: An occupation-measure and frank--wolfe approach.
\newblock In {\em Proceedings of the American Control Conference}, 2026.
\newblock Accepted for publication.

\end{thebibliography}



\appendix
\setlength{\abovedisplayskip}{3pt}
\setlength{\belowdisplayskip}{3pt}
\setlength{\abovedisplayshortskip}{1pt}
\setlength{\belowdisplayshortskip}{1pt}
\addtolength{\baselineskip}{-10pt}
\section{PDE Interpretation of the Measure Formulation}\label{sec:PDE}

\begin{Lemma}[Disintegration]\label{lem:disintegration}
Let $(\mu,\nu)\in\Delta$ be feasible. 
Then there exists a narrowly measurable family of probability measures 
$\{\rho_t\}_{t\in[0,T]}\subset\mathcal P(\mathcal X)$ 
and a measurable Markov kernel $\lambda_{t,x}$ from $[0,T]\times\mathcal X$ to $\mathcal U$
such that, for all $\varphi\in C(\Sigma)$,
\begin{align}\label{eq:disint-test}
&\int_\Sigma \varphi(t,x,u)\,d\mu(t,x,u)
\nonumber \\ & \qquad =
\int_0^T \int_{\mathcal X}\int_{\mathcal U}
\varphi(t,x,u)\, d\lambda_{t,x}(u)\, d\rho_t(x)\, dt,
\end{align}
and $\rho_t$ is the $x$–marginal of $\mu$ at time $t$.
\end{Lemma}

\begin{proof}
By the disintegration theorem on the product space 
$\Sigma=[0,T]\times\mathcal X\times\mathcal U$
(e.g.\ \cite[Corollary~10.4.17]{bogachev2007measure}), there exist
a finite Borel measure $m$ on $[0,T]$, a narrowly measurable family
$\{\rho_t\}_{t\in[0,T]}\subset\mathcal P(\mathcal X)$, and a measurable 
Markov kernel $\lambda_{t,x}$ from $[0,T]\times\mathcal X$ to $\mathcal U$ such that,
for all $\varphi\in C(\Sigma)$,
\[
\int_\Sigma \varphi\,d\mu
=
\int_0^T \int_{\mathcal X}\int_{\mathcal U}
\varphi(t,x,u)\, d\lambda_{t,x}(u)\, d\rho_t(x)\,m(dt),
\]
where $m$ is the time marginal (projection) of $\mu$ on $[0,T]$.

Since $(\mu,\nu)\in\Delta$, take in the Liouville constraint~\eqref{eq:feasibleSet}
test functions of the form $v(t,x)=\psi(t)$ with $\psi\in C^1([0,T])$.
Then $\nabla_x v\equiv0$, and feasibility gives
\[
\psi(T)-\psi(0) 
= \int_\Sigma \partial_t v\,d\mu
= \int_\Sigma \psi'(t)\,d\mu
= \int_0^T \psi'(t)\,m(dt).
\]
On the other hand,
\(
\psi(T)-\psi(0) = \int_0^T \psi'(t)\,dt.
\)
This implies $m(dt)=dt$. Substituting this into the disintegration identity yields \eqref{eq:disint-test}.
\end{proof}

\begin{Lemma}[Continuity Equation Representation]\label{lem:superposition}
Let $(\mu,\nu)\in\Delta$ be feasible and let 
$\{\rho_t\}_{t\in[0,T]}$ and $\lambda_{t,x}$ be given by 
Lemma~\ref{lem:disintegration}. Define
\[
F(t,x) := \int_{\mathcal U} f(x,u)\,d\lambda_{t,x}(u),
\qquad (t,x)\in[0,T]\times\mathcal X.
\]
\begin{itemize}
\item[(i)] $(\rho_t)_{t\in[0,T]}$ solves the continuity equation
\begin{equation}\label{eq:CE}
\partial_t \rho_t + \nabla_x\!\cdot(F(t,x)\rho_t)=0
\end{equation}
\item[(ii)] Let $\Gamma_x := AC([0,T];\mathcal X)$.
There exists a probability measure $\eta$ on $C([0,T];\mathcal X)$,
concentrated on $\Gamma_x$ (i.e., $\eta(\Gamma_x)=1$), such that for any test
functions $\varphi \in C([0,T]\times\mathcal X\times\mathcal U)$ and
$\psi \in C(\mathcal X)$,
\begin{align}\label{eq:measure_rep}
\int_{\Sigma} \varphi \, d\mu
&= \int_{\Gamma_x}
\left[ \int_0^T \int_{\mathcal U}
\varphi(t,\gamma(t),u)\, d\lambda_{t,\gamma(t)}(u)\, dt \right] d\eta(\gamma), \nonumber\\
\int_{\mathcal X} \psi \, d\nu
&= \int_{\Gamma_x} \psi(\gamma(T)) \, d\eta(\gamma).
\end{align}
Moreover, $\eta$ is consistent with the initial distribution $\rho_0$
(i.e., $(e_0)_\#\eta = \rho_0$) and is concentrated on integral curves of $F$,
meaning that for $\eta$-almost every $\gamma\in\Gamma_x$,
\begin{equation}\label{eq:ODE_constraint}
\dot{\gamma}(t)=F(t,\gamma(t)) \quad \text{for a.e. } t\in[0,T].
\end{equation}
\end{itemize}
\end{Lemma}
\begin{proof}
(i) Let $\{\rho_t\}$ and $\lambda_{t,x}$ be given by Lemma~\ref{lem:disintegration}, and define $F(t,x) := \int_{\mathcal U} f(x,u)\,d\lambda_{t,x}(u)$. From~\eqref{eq:disint-test} for all $v\in C^1([0,T]\times\mathcal X)$ we have 
\begin{align*}
&\int_\Sigma \big(\partial_t v(t,x) + \nabla_xv(t,x)\cdot f(x,u)\big)\,d\mu \\
&\quad= \int_0^T\!\!\int_{\mathcal X}\!\!\int_{\mathcal U}
\big(\partial_tv + \nabla_xv\cdot f(x,u)\big)
\,d\lambda_{t,x}(u)\,d\rho_t(x)\,dt \\
&\quad= \int_0^T\!\!\int_{\mathcal X}
\big(\partial_tv(t,x) + \nabla_xv(t,x)\cdot F(t,x)\big)
\,d\rho_t(x)\,dt.
\end{align*} 
By feasibility of $(\mu,\nu)\in\Delta$,
\begin{align*}
    \int_\Sigma \big(\partial_t v &+ \nabla_xv\cdot f(x,u)\big)\,d\mu \\ &= \int_{\mathcal X} v(T,x)\,d\rho_T(x) - \int_{\mathcal X} v(0,x)\,d\rho_0(x).
\end{align*} 
Thus,
\begin{align*}
\int_0^T\!\!\int_{\mathcal X}
\big(\partial_t v &+ \nabla_x v\cdot F\big)\,d\rho_t\,dt \\
&= \int_{\mathcal X} v(T,x)\,d\rho_T(x) - \int_{\mathcal X} v(0,x)\,d\rho_0(x).
\end{align*} 
This is exactly the weak formulation of the continuity equation~\eqref{eq:CE}. 
Therefore $(\rho_t)$ solves \eqref{eq:CE}. Since $F$ is bounded, the weak solution $(\rho_t)$ to~\eqref{eq:CE} admits a narrowly continuous representative on $[0,T]$; see \cite[Lemma~8.1.2]{ambrosio2005gradient}. We identify $\rho_t$ with this representative for the remainder of the discussion.

(ii) By our standing assumptions on $f$ and compactness of 
$\mathcal X,\mathcal U$, the vector field $F$ satisfies the hypotheses of 
\cite[Theorem~8.2.1]{ambrosio2005gradient}. Since $(\rho_t)$ solves 
\eqref{eq:CE}, we can apply the probabilistic representation theorem 
\cite[Theorem~8.2.1]{ambrosio2005gradient} to $(\rho_t)$ and $F$, which 
yields the existence of a probability measure $\eta$ on $\Gamma_x$ concentrated on the integral curves of $F$ such that $\rho_t = (e_t)_\#\eta$ for all $t$ satisfying~\eqref{eq:measure_rep}, 
proving the assertion. 
\end{proof}

\section{Approximation and Relaxation Results}
\begin{Lemma}[Approximation by Classical Trajectories]\label{lem:approximation}
Let $\gamma \in \Gamma_x$ be an integral curve of the averaged dynamics
defined in Lemma~\ref{lem:superposition}, generated by the relaxed control family
$\{\lambda_{t,\gamma(t)}\}_{t\in[0,T]}$.
Define the augmented cost state variable $z:[0,T]\to\mathbb R$ by
\begin{equation}\label{eq:z_relaxed}
\dot{z}(t)=\int_{\mathcal U} g_{\mu_k}(t,\gamma(t),u)\,d\lambda_{t,\gamma(t)}(u),
\qquad z(0)=0,
\end{equation}
where $g_{\mu_k}$ is defined in Theorem~\ref{thm:firstVariation}.
Then there exists a sequence of measurable controls $u_j:[0,T]\to\mathcal U$ and
corresponding classical trajectories $x_j\in AC([0,T];\mathcal X)$ with $\dot x_j(t)=f(x_j(t),u_j(t))$ for a.e. $t\in[0,T]$, $x_j(0)=\gamma(0)$, such that, letting $y_j$ be defined by
\begin{equation}\label{eq:z_classical}
\dot y_j(t)=g_{\mu_k}(t,x_j(t),u_j(t)),\qquad y_j(0)=0,
\end{equation}
we have the uniform convergence
\begin{equation}\label{eq:uniform_conv_aug}
\lim_{j\to\infty}\sup_{t\in[0,T]}
\Big(|x_j(t)-\gamma(t)|+|y_j(t)-z(t)|\Big)=0.
\end{equation}
\end{Lemma}
\begin{proof}
Consider the augmented state space $\mathcal X\times\mathbb R$ and define
\begin{align*}
\tilde f(t,x,u)&:=\big(f(x,u)^\top,\; g_{\mu_k}(t,x,u)\big)^\top, \\
G(t,x)&:=\{\tilde f(t,x,u):u\in\mathcal U\}.
\end{align*}
By construction of $z$ in~\eqref{eq:z_relaxed} and the averaged dynamics in
Lemma~\ref{lem:superposition}, the pair $(\gamma,z)$ with initial condition
$(\gamma(0),z(0))=(\gamma(0),0)$ satisfies the convexified differential inclusion
\[
(\dot\gamma(t),\dot z(t)) \in \operatorname{co}\,G(t,\gamma(t))
\quad \text{for a.e. } t\in[0,T],
\]
where $\operatorname{co} G(t,x)$ denotes the convex hull of $G(t,x)$ and hence is a relaxed $G$--trajectory in the terminology of
\cite[Theorem~2.7.2]{vinter2010optimal}.
Applying \cite[Theorem~2.7.2]{vinter2010optimal} to the multifunction $G$
yields a sequence of ordinary $G$--trajectories $(x_j,y_j)$ with the
same initial condition such that $(x_j,y_j)\to(\gamma,z)$ uniformly on
$[0,T]$. Since $G(t,x)=\{\tilde f(t,x,u):u\in\mathcal U\}$, each ordinary
$G$--trajectory admits a measurable control representation
$\dot x_j=f(x_j,u_j)$ and $\dot y_j=g_{\mu_k}(t,x_j,u_j)$ for some measurable
$u_j:[0,T]\to\mathcal U$. This proves~\eqref{eq:uniform_conv_aug}.
\end{proof}

\begin{Corollary}[Lower Bound for FW Subproblem]\label{cor:fw_lower_bound}
Let $\gamma \in \Gamma_x$ be as in Lemma~\ref{lem:approximation} with initial state $\xi = \gamma(0)$. The cost along $\gamma$ is lower-bounded by the value function of the classical parametric problem:
\begin{align}\label{eq:fw_inequality}
&C(\gamma):=\int_0^T \int_{\mathcal U} g_{\mu_k}(t,\gamma(t),u)\,d\lambda_{t,\gamma(t)}(u)\,dt + \Psi(\gamma(T)) \nonumber \\
& \ge 
\inf_{\substack{x(0)=\xi \\ \dot{x}=f(x,u)}} \left\{ \int_0^T g_{\mu_k}(t,x,u)\,dt + \Psi(x(T)) \right\}.
\end{align}
\end{Corollary}
\begin{proof}
Let $\{(x_j, u_j, y_j)\}$ be the approximating sequence from Lemma~\ref{lem:approximation}. 
The uniform convergence established in \eqref{eq:uniform_conv_aug} implies the convergence of the endpoints: $y_j(T) \to z(T)$ and $x_j(T) \to \gamma(T)$. 
Consequently, the total cost of the classical trajectories converges to the generalized cost:
\[
\lim_{j\to\infty} \big( y_j(T) + \Psi(x_j(T)) \big) = z(T) + \Psi(\gamma(T)) = C(\gamma).
\]
Since the cost of every feasible classical trajectory $(x_j, u_j)$ is bounded below by the infimum, the inequality \eqref{eq:fw_inequality} follows immediately by taking the limit $j\to\infty$.
\end{proof}

\section{Proof of Theorem~\ref{thm:lmo}}\label{sec:appendix_lmo_proof}
Let $(\mu, \nu) \in \Delta$ be any feasible solution. By Lemma~\ref{lem:superposition} and representation \eqref{eq:measure_rep}, there exists a probability measure $\eta$ on the path space $\Gamma_x$ such that:
\[
\langle g_{\mu_k}, \mu \rangle + \langle \Psi, \nu \rangle = \int_{\Gamma_x} C(\gamma) \, d\eta(\gamma),
\]
where $C(\gamma)$ is defined in Corollary~\ref{cor:fw_lower_bound}. The measure $\eta$ is consistent with the initial distribution, satisfying $(e_0)_\#\eta = \rho_0$.

By applying the pointwise lower bound $C(\gamma) \ge V(\gamma(0))$ from Corollary~\ref{cor:fw_lower_bound} to $\eta$-almost every trajectory, we obtain the following sequence of relations:
\begin{align*}
&\langle g_{\mu_k}, \mu \rangle + \langle \Psi, \nu \rangle 
= \int_{\Gamma_x} C(\gamma) \, d\eta(\gamma) \\
&\ge \int_{\Gamma_x} V(\gamma(0)) \, d\eta(\gamma)= \int_{\mathcal X} V(\xi) \, d\rho_0(\xi),
\end{align*}
where the final equality follows from the change of variables under the initial evaluation map $e_0(\gamma) = \gamma(0)$.

By construction, the aggregated measures $(\bar\mu_k, \bar\nu_k)$ are formed by a measurable selection of classical trajectories $\omega_k(\xi)$ that attain the minimum $V(\xi)$ for $\rho_0$-almost every $\xi$. By the linearity of the integral, their exact cost is:
\[
\langle g_{\mu_k}, \bar\mu_k \rangle + \langle \Psi, \bar\nu_k \rangle = \int_{\mathcal X} V(\xi) \, d\rho_0(\xi).
\]
Combining this with the global lower bound established above, we conclude that for all feasible $(\mu,\nu) \in \Delta$,
\[
\langle g_{\mu_k}, \mu \rangle + \langle \Psi, \nu \rangle \ge \langle g_{\mu_k}, \bar\mu_k \rangle + \langle \Psi, \bar\nu_k \rangle,
\]
which proves that $(\bar\mu_k, \bar\nu_k)$ is a minimizer for the FW subproblem. \hfill$\square$

\section{Proof of Theorem~\ref{thm:fw-rate}}\label{app:fw-proof}
Let $z_k:=(\mu_k,\nu_k)\in\Delta$ and let $s_k:=(\bar\mu_k,\bar\nu_k)\in\Delta$
be the solution returned by the FW linear minimization oracle, so that
$z_{k+1}=(1-\alpha_k)z_k+\alpha_k s_k$.
By $L$-smoothness of $J$ on $\Delta$ (Definition~\ref{def:Lsmooth}),
\begin{align*}
J(z_{k+1})-J(z_k)
&\le \delta J(z_k;\,z_{k+1}-z_k)+\frac{L}{2}\|z_{k+1}-z_k\|^2 \\
&=\alpha_k\,\delta J(z_k;\,s_k-z_k)+\frac{L}{2}\alpha_k^2\|s_k-z_k\|^2 \\
&\le \alpha_k\,\delta J(z_k;\,z^*-z_k)+\frac{L}{2}\alpha_k^2\|s_k-z_k\|^2 \\
&\le \alpha_k\,(J^*-J(z_k))+\frac{L}{2}\alpha_k^2\,(2(T+1))^2,
\end{align*}
where the first inequality is $L$-smoothness, the equality uses $z_{k+1}-z_k=\alpha_k(s_k-z_k)$,
the second inequality uses optimality of $s_k$ for the linear subproblem, the third inequality
uses convexity ($\delta J(z_k;z^*-z_k)\le J(z^*)-J(z_k)$), and the last term follows from
\[
\|s_k-z_k\|\le \|\mu[\omega_k]-\mu_k\|_{\mathrm{TV}}+\|\nu[\omega_k]-\nu_k\|_{\mathrm{TV}}
\le 2(T+1).
\]
Let $\Delta_k:=J(z_k)-J^*$. Then
\[
\Delta_{k+1}\le (1-\alpha_k)\Delta_k + 2L(T+1)^2\,\alpha_k^2.
\]
With $\alpha_k=\frac{2}{k+2}$, a standard induction yields
\[
\Delta_k \le \frac{8L(T+1)^2}{k+2}, \qquad k\ge 1,
\]
which completes the proof. \hfill$\square$

\end{document}